\documentclass[12pt,dvipdfmx,a4paper]{article}
\usepackage[margin=25truemm]{geometry}
\usepackage[colorlinks=true,linkcolor=blue,citecolor=blue]{hyperref}
\usepackage{amssymb,amsthm,amsmath, amsfonts,ascmac, enumerate}
\usepackage{stmaryrd} 
\usepackage{mathtools} 
\usepackage{bm}
\numberwithin{equation}{section}

\theoremstyle{plain}
\newtheorem{theorem}{Theorem}[section]
\newtheorem{lemma}[theorem]{Lemma}
\newtheorem{proposition}[theorem]{Proposition}
\newtheorem{corollary}[theorem]{Corollary}

\theoremstyle{definition}
\newtheorem{definition}[theorem]{Definition}

\newtheorem{example}[theorem]{Example}

\newtheorem{remark}[theorem]{Remark}

\title{Convolution, cumulants and infinitesimal generators in the formal power series ring}
\author{Shuhei Tsujie \ and \  Yuki Ueda
}

\date{\today}

\begin{document}
\maketitle

\abstract{
We extend the notions of finite free convolution and finite free cumulants to the setting of formal power series by introducing their natural analogues, namely $t$-deformed convolution and $t$-deformed cumulants. In this framework, we establish $t$-deformed analogues of the law of large numbers and the central limit theorem, revealing structural parallels with classical, free, and finite free probability theories. 

We show that the case $t=-1$ recovers classical convolution at the level of moment generating functions, thereby connecting the theory directly to classical probability. 

We further investigate the infinitesimal generators associated with $\boxplus^t$-continuous semigroups, deriving explicit representation formulas that clarify how these generators describe the infinitesimal evolution of the semigroup. 
In the case $t = d$, our results yield explicit formulas for finite free infinitesimal generators. 
In the case $t = -1$, we relate these generators to those of one-dimensional L\'{e}vy processes by identifying the corresponding terms in their representations. 
This establishes a direct connection between $\boxplus^t$-convolution semigroups and classical L\'{e}vy--Khintchine-type generators.
}
	
\vspace{3mm}



\tableofcontents

\section{Introduction}

\subsection{Finite free convolution}
Let $\mathbb{R}[x]$ denote the polynomial ring over $\mathbb{R}$ and $\mathbb{R}[x]_{d}$ its subspace consisting of polynomials of degree at most $d$. 
For $f(x) = \sum_{i=0}^{d}a_{i}x^{d-i}$ and $g(x) = \sum_{i=0}^{d}b_{i}x^{d-i}$ in $\mathbb{R}[x]_{d}$, we define the \emph{finite free convolution} (or \emph{polynomial convolution}) of $f(x)$ and $g(x)$ by 
\begin{align*}
(f \boxplus_{d} g)(x) \coloneqq \sum_{k=0}^{d}\sum_{i+j=k}\dfrac{(d)_{\underline{k}}}{(d)_{\underline{i}}(d)_{\underline{j}}}a_{i}b_{j}x^{d-k}, 
\end{align*}
where $(d)_{\underline{k}} \coloneqq d(d-1) \cdots (d-k+1)$ denotes the Pochhammer symbol (falling factorial). 

In recent years, this convolution has also appeared in the study of characteristic polynomials of (random) matrices by Marcus, Spielman, and Srivastava \cite{marcus2021polynomial, marcus2022finite-ptarf}. More precisely, let $p$ and $q$ be characteristic polynomials of $d$-dimensional Hermitian matrices $A$ and $B$, respectively. Then the finite free convolution $p\boxplus_d q$ coincides with the expected characteristic polynomial of $A+UBU^\ast$, where $U$ is a random unitary matrix distributed according to the Haar measure on the unitary group $\mathcal{U}(d)$. This representation mirrors the random matrix model underlying free additive convolution, in which freely independent random variables arise as large-dimensional limits of unitarily invariant random matrices. In this sense, the finite free convolution $\boxplus_d$ can be viewed as a finite-dimensional analogue of the free additive convolution $\boxplus$, which describes the distribution of the sum of freely independent random variables, introduced by Voiculescu \cite{voiculescu1986addition, voiculescu2006symmetries}.  Actually, the finite free convolution $\boxplus_d$ approximates the free additive convolution $\boxplus$ as the degree $d$ tends to infinity; see \cite{marcus2021polynomial}. Furthermore, Arizmendi and Perales \cite[Corollary 5.5]{arizmendi2018cumulants-joctsa} provided a combinatorial proof of this convergence by introducing {\it finite free cumulants}, which serve as finite-dimensional analogues of the combinatorial structures arising in free probability. For details of combinatorial theory of free probability, we refer to the monograph by Nica and Speicher \cite{nica2006lectures}.

In the following, we introduce finite free cumulants and review the relevant results from \cite{arizmendi2018cumulants-joctsa}. Define 
$$
\mathbb{R}[x]_{{\rm monic},d} := \left\{f(x) = \sum_{k=0}^{d}a_{k}x^{d-k}\in \mathbb{R}[x]_{d}: a_{0}=1 \right\} .
$$
For $f \in \mathbb{R}[x]_{{\rm monic},d} $, the \emph{finite free cumulants} $\kappa_{1}^{(d)}(f), \dots, \kappa_{d}^{(d)}(f)$ are defined to be the numbers satisfying 
\begin{align*}
\sum_{i=1}^{d}\kappa_{i}^{(d)}(f)z^{i} \equiv -\dfrac{z}{d}\dfrac{d}{dz}\log\left(1 + \sum_{k=1}^{d}\dfrac{d^{k}}{(d)_{\underline{k}}}a_{k}z^{k}\right) \pmod{z^{d+1}}, 
\end{align*}
where $\log$ stands for the formal logarithm. 
Namely, 
\begin{align*}
\log(1+z) = \sum_{k=1}^{\infty}\dfrac{(-1)^{k-1}}{k}z^{k}. 
\end{align*}

Let $\lambda_{1}, \dots, \lambda_{d} \in \mathbb{C}$ denote the roots of $f(x)$. 
Define the $i$-th \emph{power sum} $p_{i}(f)$ by 
\begin{align*}
p_{i}(f) \coloneqq \lambda_{1}^{i} + \dots + \lambda_{d}^{i}. 
\end{align*}
Note that $p_{i}(f)$ is an equivalent notion to the $i$-th moment $\frac{1}{d}(\lambda_{1}^{i} + \dots + \lambda_{d}^{i})$ of the roots of $f(x)$. 

For $r \in \mathbb{R}$, the \emph{dilation} $D_{r}f(x)$ is defined by 
\begin{align*}
D_{r}f(x) \coloneqq \prod_{i=1}^{d}(x-r\lambda_{i}). 
\end{align*}

The finite free cumulants satisfy the following axioms that are analogous to the axioms for free cumulants, introduced by Lehner \cite{lehner2002free-ejoc}. 

\begin{theorem}[Arizmendi and Perales {\cite[Proposition 3.3]{arizmendi2018cumulants-joctsa}}]\label{Arizmensi-Perales}
Given $d \in \mathbb{Z}_{\ge1}$, the finite free cumulants $\kappa_{1}^{(d)}, \dots, \kappa_{d}^{(d)}$ are unique maps satisfying the following conditions. 
\begin{enumerate}[\rm (i)]
\item $\kappa_{i}^{(d)}(f)$ is a polynomial in $p_{1}(f), \dots, p_{i}(f)$ with leading term $\frac{d^{i-1}}{(d)_{\underline{i}}} p_{i}(f)$. 
That is, there exists a polynomial $F_{i}$ in $i-1$ variables such that 
$$
\kappa_{i}^{(d)}(f) = \frac{d^{i-1}}{(d)_{\underline{i}}}p_{i}(f) + F_{i}(p_{1}(f), \dots, p_{i-1}(f)).
$$ 
\item Homogeneity: $\kappa_{i}^{(d)}(D_{r}f) = r^{i}\kappa_{i}^{(d)}(f)$ for any $r \in \mathbb{R}$ and $i \in \mathbb{Z}_{\ge1}$.
\item Additivity: $\kappa_{i}^{(d)}(f \boxplus_{d} g) = \kappa_{i}^{(d)}(f) + \kappa_{i}^{(d)}(g)$. 
\end{enumerate}
\end{theorem}

According to \cite[Theorem 5.4]{arizmendi2018cumulants-joctsa}, let $f_d \in \mathbb{R}[x]_{{\rm monic},d}$ be a real-rooted polynomial of degree $d$ whose empirical root distribution 
$$
\mathfrak{m}\llbracket f_d\rrbracket \coloneqq \frac{1}{d}\sum_{x: f_d(x)=0}\delta_x
$$ 
converges weakly to some compactly supported probability measure $\mu$ on $\mathbb{R}$ as $d\to \infty$. Then the finite free cumulants $\kappa_n^{(d)}(f_d)$ of $f_d$ converge to the free cumulants $\kappa_n(\mu)$ of $\mu$ as $d\to \infty$. Due to the result, if $f_d,g_d\in \mathbb{R}[x]_{{\rm monic},d}$ are real-rooted polynomials such that $\mathfrak{m}\llbracket f_d\rrbracket\xrightarrow{w} \mu$ and $\mathfrak{m}\llbracket g_d\rrbracket \xrightarrow{w}\nu$ for some compactly supported probability measures $\mu,\nu$ on $\mathbb{R}$, then 
$$
\mathfrak{m}\llbracket f_d \boxplus_d g_d\rrbracket \xrightarrow{w} \mu\boxplus \nu \qquad \text{as} \quad  d\to \infty,
$$
see \cite[Corollary 5.5]{arizmendi2018cumulants-joctsa}. More recently, Fujie \cite{fujie2026regularity} proved this convergence without the assumption that the limiting probability measures are compactly supported.

\subsection{$t$-deformed convolution}

In this paper, we extend the definition and basic properties of finite free convolution and finite free cumulants from polynomials to the ring of formal power series. This extension leads to a one-parameter family of convolutions, called the $t$-deformed convolution, together with the corresponding notion of $t$-deformed cumulants.

To motivate this construction, we embed polynomials into the ring of formal power series. Define a map $\iota_{d} \colon \mathbb{R}[x]_{d} \to \mathbb{R}\llbracket z \rrbracket$ by 
$$
\iota_{d}(f)(z) \coloneqq z^{d}f(z^{-1}), 
$$
where
\begin{align*}
\mathbb{R}\llbracket z \rrbracket \coloneqq \left\{\sum_{k=0}^\infty a_{k} z^{k} : a_{k} \in \mathbb{R} \right\}
\end{align*}
is the ring of formal power series. This embedding allows us to reinterpret finite free convolution in the setting of formal power series.

The convolution $\boxplus_{d}$ on $\mathbb{R}[x]_d$ induces a convolution on $\operatorname{Im}\iota_{d}$. For $A(z) = \sum_{k=0}^{d} a_{k} z^{k}$ and $B(z) = \sum_{k=0}^{d} b_{k} z^{k}$ in $\operatorname{Im}\iota_{d}$, we define
\begin{align}\label{d-convolution}
(A \boxplus^{d} B)(z) \coloneqq \sum_{k=0}^{d} \sum_{i+j=k} \frac{(d)_{\underline{k}}}{(d)_{\underline{i}} (d)_{\underline{j}}} a_{i} b_{j} z^{k}.
\end{align}
Then we obtain
\begin{align}\label{free-t-convolution}
\iota_d(f \boxplus_{d} g) = \iota_{d}(f) \boxplus^{d} \iota_{d}(g)
\end{align}
for all $f, g \in \mathbb{R}[x]_{d}$. 
However, although $\operatorname{Im}\iota_{d} \subset \operatorname{Im}\iota_{d+1}$, the operations $\boxplus^{d}$ are not compatible as $d$ varies. This lack of compatibility motivates the construction of a unified convolution on the whole ring $\mathbb{R}\llbracket z \rrbracket$.

\begin{definition}
Let $A(z) = \sum_{k=0}^{\infty} a_{k} z^{k}$ and $B(z) = \sum_{k=0}^{\infty} b_{k} z^{k}$ be formal power series in $\mathbb{R}\llbracket z \rrbracket$, and let $t \in \mathbb{R} \setminus \mathbb{Z}_{\geq 0}$. 
We define the \emph{$t$-deformed convolution} by
\begin{align*}
(A \boxplus^{t} B)(z) \coloneqq \sum_{k=0}^{\infty} \sum_{i+j=k} \frac{(t)_{\underline{k}}}{(t)_{\underline{i}} (t)_{\underline{j}}} a_{i} b_{j} z^{k}.
\end{align*}
For $t = d \in \mathbb{Z}_{\ge 1}$, the $d$-deformed convolution is defined by \eqref{d-convolution}.
\end{definition}

The parameter $t$ provides a continuous deformation of finite free convolution. In particular, for $A, B \in \operatorname{Im}\iota_{d}$, we have
\[
\lim_{t \to d} (A \boxplus^{t} B)(z) = (A \boxplus^{d} B)(z).
\]
Moreover, $\boxplus^t$ is a bilinear, commutative, and associative operation for all $t \in \mathbb{R}^\times$.

We next describe the main structural consequences of this construction. In Section \ref{sec2}, we introduce \textit{$t$-deformed cumulants} $\{\kappa_n^{t}(A)\}_{n\in \mathbb{Z}_{\ge1}}$, which linearize the convolution:
$$
\kappa_n^t \bigl(A \boxplus^t B\bigr) = \kappa_n^t(A) + \kappa_n^t(B), \qquad n \in \mathbb{Z}_{\ge1}.
$$
Thus, the $t$-deformed cumulants play the same role as classical, free, and finite free cumulants. We also establish analogues of the law of large numbers and the central limit theorem, showing that $\boxplus^t$ admits a probabilistic interpretation parallel to finite free probability; see Theorem \ref{limit_thms}.

In Section \ref{sec:3}, we show that the case $t = -1$ connects the theory to classical probability. If $X$ and $Y$ are independent random variables, then
\[
M_{X+Y}(z) = \bigl(M_X \boxplus^{-1} M_Y \bigr)(z),
\]
where $M_X$ denotes the ordinary generating function of moments of $X$. This identity shows that $\boxplus^{-1}$ corresponds to classical convolution at the level of moments. We further relate $t$-deformed cumulants to classical cumulants via a differential identity; see Theorem \ref{Classicalresult}.

Taken together, these results show that $\boxplus^t$ provides a unified framework interpolating between different notions of independence. More precisely, it recovers classical convolution at $t = -1$, finite free convolution at $t = d \in \mathbb{Z}_{\ge 1}$, and free convolution in the limit $t =d\to \infty$. Therefore, the $t$-deformed convolution establishes a continuous bridge between classical, finite free, and free probability theories.

\subsection{Infinitesimal generators for $\boxplus^t$-semigroups}

In Section \ref{sec:4}, we study the infinitesimal generators of $\boxplus^t$-continuous semigroups of formal power series in order to understand the dynamical structure of the $\boxplus^t$-time evolution. The infinitesimal generator plays a central role in describing how a process evolves over an infinitesimal time scale, and thus provides a fundamental tool for analyzing its global behavior.

In classical probability theory, infinitesimal generators are essential for understanding stochastic processes. In particular, they characterize the time evolution of distributions via partial differential equations such as the Kolmogorov forward equation. For example, the infinitesimal generator of one-dimensional Brownian motion $\{B(s)\}_{s \ge 0}$ is the Laplacian $\frac{1}{2}\Delta$, and the corresponding density function $p(s,x)$ satisfies the heat equation 
$$
\partial_s p = \frac{1}{2}\Delta p. 
$$

This connection motivates the study of analogous structures in the setting of formal power series and $t$-deformed convolution $\boxplus^t$. Let $\mathfrak{Q} = \{Q_s\}_{s \ge 0}$ be a $\boxplus^t$-continuous semigroup of formal power series. Then the associated operators $\{T_s = T_{Q_s}\}_{s \ge 0}$ form a semigroup under composition, where
\begin{align*}
T_s A := A \boxplus^t Q_s, \qquad s \ge 0,
\end{align*}
for a formal power series $A$. For the semigroup $\{T_s\}_{s\ge0}$, we define the {\it $(t,r)$-deformed infinitesimal generator} $\mathcal{L}_r^{(t)}[\mathfrak{Q}]$ by
$$
\mathcal{L}_r^{(t)}[\mathfrak{Q}] A \coloneqq  \lim_{s\to 0^+} \frac{T_s A -A}{s}, \qquad A\in \mathcal{D}(\mathcal{L}_r^{(t)}[\mathfrak{Q}] )
$$
where $ \mathcal{D}(\mathcal{L}_r^{(t)}[\mathfrak{Q}] )$ denotes the set of power series $A$ for which the limit exists; see Definition \ref{def:generators}. Then the infinitesimal generator $\mathcal{L}_r^{(t)}[\mathfrak{Q}]$ is conjugate to the multiplication operator of 
\begin{align*}
\eta_\mathfrak{Q}(z) \coloneqq \lim_{s\to 0^+} \frac{\Phi_t(Q_s)(z)-1}{s} \quad (|z|<r),
\end{align*}
via the transform $\Phi_t$, where
\begin{align*}
\Phi_t\left( \sum_{k=0}^\infty a_k z^k \right)
\coloneqq \sum_{k=0}^\infty \frac{a_k}{(t)_{\underline{k}}} z^k.
\end{align*}
This representation reveals a simple and tractable structure underlying the $\boxplus^t$-dynamics; see Theorem \ref{thm:generators}.

Next, we consider the case $t=d$, which is closely related to finite free probability. Let $\mathfrak{q} = \{q_s\}_{s \ge 0} \subset \mathbb{R}[x]_d$ be a $\boxplus_d$-semigroup, that is, $q_0(x)=x^d$ and $q_{s_1}\boxplus_d q_{s_2}=q_{s_1+s_2}$ for any $s_1,s_2\ge0$. For each $s \ge 0$, define the convolution operator $T_{q_s}$ by
$$
T_{q_s} p :=p \boxplus_d q_s, \qquad p \in \mathbb{R}[x]_d
$$ 
and the {\it finite free infinitesimal generator} by
$$
\mathcal{L}[\mathfrak{q}] f \coloneqq \lim_{s\to 0^+} \frac{T_{q_s} f -f}{s}, \qquad  f\in \mathbb{R}[x]_d
$$
whenever the limit exists. We now relate this operator to the $d$-deformed infinitesimal generator. More precisely, we have
\begin{align*}
\iota_d\bigl(\mathcal{L}[\mathfrak{q}]f\bigr) = \mathcal{L}^{(d)}\bigl[\iota_d(\mathfrak{q}) \bigr]\bigl( \iota_d(f)\bigr), \qquad f\in \mathbb{R}[x]_d,
\end{align*}
where $\mathcal{L}^{(d)}$ denotes the $d$-deformed infinitesimal generator (see Section \ref{sec:4.3}) and $\iota_d(\mathfrak{q}) = (\iota_d(q_s))_{s\ge0}$. This connection allows us to interpret finite free infinitesimal generators within the $(t,r)$-deformed framework.

In the final part of Section~\ref{sec:4}, we consider the special case $t = -1$, where the theory connects directly to classical probability. Let $\{X(s)\}_{s \ge 0}$ be a L\'{e}vy process such that the moment generating function is finite in a neighborhood of the origin, and let $\mathfrak{M} = \{M_{X(s)}\}_{s \ge 0}$ denote the corresponding ordinary generating functions of moments. Then the $(-1,r)$-deformed infinitesimal generator $\mathcal{L}_r^{(-1)}[\mathfrak{M}]$ is conjugate to the multiplication operator by the L\'{e}vy--Khintchine representation
\begin{align*}
\eta_{\gamma,a,\nu}(z)
 = -\gamma z + \frac{a}{2}z^2 + \int_{\mathbb{R}}\left(e^{-zx}-1+zx\mathbf{1}_{|x|\le 1}\right)\, \nu(dx),
\end{align*}
for $\gamma \in \mathbb{R}$, $a \ge 0$, and a L\'{e}vy measure $\nu$, via the transform $\Phi_{-1}$; see Theorem \ref{eq:semigroup-Levy}. 

Since the infinitesimal generator admits the L\'{e}vy--Khintchine representation, the associated evolution equation can be interpreted as an analogue of the Kolmogorov forward (or Fokker--Planck) equation for L\'{e}vy processes. In general, the generator contains a nonlocal integral term, and hence the corresponding evolution equation takes the form of an integro-differential equation.
In the formal power series setting, the $\boxplus^{-1}$-evolution $F(s,z) = \bigl(A \boxplus^{-1} M_{X(s)}\bigr)(z)$ satisfies the {\it $(-1,r)$-deformed Kolmogorov forward equation}
\[
\frac{\partial}{\partial s} \Phi_{-1}\bigl(F(s,\cdot) \bigr)(z) = \eta_{\gamma,a,\nu}(z) \Phi_{-1}\bigl(F(s,\cdot)\bigr)(z).
\]
In the special case where $\{X(s)\}_{s\ge0}$ is a one-dimensional Brownian motion, the generator reduces to the multiplication operator $z^2/2$. Consequently, the evolution equation simplifies to
\[
\frac{\partial}{\partial s} \Phi_{-1}\bigl(F(s,\cdot)\bigr)(z)
=
\frac{z^2}{2} \, \Phi_{-1}\bigl(F(s,\cdot)\bigr)(z),
\]
with $F(0,z)=A(z)$. Thus, the evolution is governed by the {\it $(-1,r)$-deformed heat equation}, corresponding to the classical heat equation under the transform $\Phi_{-1}$.

These results show that the $\boxplus^t$-framework admits evolution equations analogous to the evolution equations associated with classical stochastic processes, with the infinitesimal generator governing the dynamics. Moreover, the $t$-deformation provides a unified perspective: it recovers classical convolution at $t = -1$, finite free convolution at integer values of $t$, and connects to free probability in the limit $t \to \infty$. In this way, the $(t,r)$-deformed infinitesimal generator highlights a structural link between classical, finite free, and free probability theories.

\section{$t$-deformed convolution and cumulants}\label{sec2}

\subsection{$t$-deformed cumulants}\label{sec2.1}
Let $U \coloneqq 1 + z \mathbb{R}\llbracket z \rrbracket$ denote the group of principal units in $\mathbb{R}\llbracket z \rrbracket$. 
We also define $U_d \coloneqq \iota_d\bigl(\mathbb{R}[x]_{\mathrm{monic},d}\bigr)$. 
Then $U_d \subset U$ for all $d \ge 1$.

For $r \in \mathbb{R}$, we define the \emph{dilation} $D_{r}[A](z)$ of $A\in U$ by 
$$
D_{r}[A](z) \coloneqq A(rz).
$$ 
Then $D_{r}[\iota_{d}(f)](z) = \iota_{d}(D_{r}f)(z)$ for $f \in \mathbb{R}[x]_{{\rm monic},d}$. 

We also define ``power sums" $\{p_{k}(A)\}_{k \in \mathbb{Z}_{\ge1}}$ for $A \in U$. 
However, it does not make sense to consider the roots of the formal power series $A(z)$. 
Instead of considering the roots, we define $\{p_{k}(A)\}_{k\in \mathbb{Z}_{\ge1}}$ using the relation between elementary symmetric functions and the power sum symmetric functions (Newton's identities), as follows. 
Let $\{x_{i}\}_{i \in \mathbb{Z}_{\ge1}}$ be countably many indeterminates. 
Define the $k$-th elementary symmetric function $e_{k}$ and the $k$-th power sum symmetric function $p_{k}$ by 
\begin{align*}
e_{k} \coloneqq \sum_{i_{1} < \dots < i_{k}}x_{i_{1}} \cdots x_{i_{k}}, \qquad
p_{k} \coloneqq \sum_{i=1}^{\infty}x_{i}^{k}. 
\end{align*}
Then the formal computation yields 
\begin{align*}
-z\dfrac{d}{dz}\log\left(1 + \sum_{k=1}^{\infty}(-1)^{k}e_{k}z^{k}\right) 
&= -z\dfrac{d}{dz}\log\left(\prod_{i=1}^{\infty}(1-x_{i}z)\right)\\
&= \sum_{i=1}^{\infty}\dfrac{x_{i}z}{1-x_{i}z}= \sum_{k=1}^{\infty}p_{k}z^{k}. 
\end{align*}
Therefore, for $A \in U$, we define $\{p_i(A)\}_{i\in \mathbb{Z}_{\ge1}}$ by
\begin{align*}
M[A](z) \coloneqq \sum_{k=1}^{\infty}p_{k}(A)z^{k} \coloneqq -z\dfrac{d}{dz} \log\bigl(A(z)\bigr).
\end{align*}

\begin{definition}
For $t\in \mathbb{R}\setminus \mathbb{Z}_{\ge0}$, we define the \emph{$t$-deformed cumulants} $\{\kappa_{i}^{t}(A)\}_{i \in \mathbb{Z}_{\ge1}}$ of $A \in U$ by
$$
C^t[A](z) \coloneqq \sum_{i=1}^\infty \kappa_i^t(A)z^i \coloneqq -\frac{z}{t} \frac{d}{dz}\log \left(1+ \sum_{k=1}^\infty \frac{t^{k}}{(t)_{\underline{k}}} a_{k} z^{k} \right).
$$
If $A\in U_d$ for some $d\in \mathbb{Z}_{\ge1}$, we define the \emph{$d$-deformed cumulants} $\{\kappa_{i}^{d}(A)\}_{i \in \mathbb{Z}_{\ge1}}$ by
$$
C^d[A](z) \coloneqq \lim_{t\to d} C^t[A](z)\coloneqq \sum_{i=1}^\infty \kappa_i^d(A)z^i  = -\frac{z}{d} \frac{d}{dz}\log\left(1+ \sum_{k=1}^d \frac{d^{k}}{(d)_{\underline{k}}} a_{k} z^{k}\right).
$$
We call $C^t[A](z)$ the \emph{$t$-deformed cumulant transform} of $A$. 
\end{definition}

In what follows, we restrict our attention to the case $t \in \mathbb{R} \setminus \mathbb{Z}_{\ge 0}$. 
Both the definitions and the results (Lemmas~\ref{lem:injective} and \ref{homogeneity}, and Theorem~\ref{t-cumulants}) extend to the case $t = d$ by taking the limit $t \to d$, 
with the corresponding objects interpreted on $U_d$.

For $t\in \mathbb{R}\setminus\mathbb{Z}_{\ge0}$ and $A\in U $, we define
\begin{align*}
E^{t}[A](z) \coloneqq 1 + \sum_{k=1}^{\infty}\frac{t^{k}}{(t)_{\underline{k}}}a_{k}z^{k}. 
\end{align*}
Then we have 
\begin{align}\label{C^t-E^t[A]}
C^{t}[A](z)= -\frac{z}{t} \frac{d}{dz}\log \bigl(E^{t}[A](z)\bigr)
\end{align}
and 
\begin{align}\label{E^t}
tC^{t}[A](z) = -z\frac{d}{dz}\log \bigl(E^{t}[A](z)\bigr) = M[E^{t}[A]](z). 
\end{align}

The $t$-deformed cumulant transform uniquely determines the formal power series.
\begin{lemma}\label{lem:injective}
The transform $C^{t} : U \to \mathbb{R}\llbracket z \rrbracket$ is injective, that is, if $C^{t}[A]=C^{t}[B]$, then $A=B$.
\end{lemma}
\begin{proof}
Suppose that $C^{t}[A]=C^{t}[B]$ for $A,B\in U$. By \eqref{C^t-E^t[A]}, we obtain
$$
\frac{d}{dz} \bigl\{\log \bigl(E^{t}[A](z)\bigr) -\log \bigl(E^{t}[B](z)\bigr) \bigr\}=0.
$$
Since $\log \bigl(E^t[A](z) \bigr)|_{z=0}=\log\bigl(E^t[B](z)\bigr)|_{z=0}=0$, it follows that
$$
\log \bigl(E^{t}[A](z)\bigr) = \log \bigl(E^{t}[B](z)\bigr),
$$
and hence $E^{t}[A](z)=E^{t}[B](z)$. Therefore, $A=B$.
\end{proof}

\begin{lemma}\label{homogeneity}
Let $A \in U$. For any $r \in \mathbb{R}$ and $k \in \mathbb{Z}_{\ge1}$, we have $p_{k}\bigl(D_{r}[A]\bigr) = r^{k}p_{k}(A)$. 
\end{lemma}
\begin{proof}
A direct computation shows that
\begin{align*}
M[D_{r}[A]](z) 
&= -z\dfrac{d}{dz}\log\bigl(D_{r}[A](z)\bigr)\\
&= -z\dfrac{d}{dz}\log\bigl(A(rz)\bigr)\\
&= -rz \dfrac{\bigl(\frac{d}{dz}A\bigr)(rz)}{A(rz)}= M[A](rz).
\end{align*}
Consequently, we get $p_{k}\bigl(D_{r}[A]\bigr) = r^{k}p_{k}(A)$ for any $k\in \mathbb{Z}_{\ge1}$. 
\end{proof}

The $t$-deformed cumulants satisfy the axiomatized conditions for cumulants as follows. 

\begin{theorem}\label{t-cumulants}
The $t$-deformed cumulants $\{\kappa_{i}^{t}(A)\}_{i \in \mathbb{Z}_{\ge1}}$ are unique maps satisfying the following conditions:
\begin{enumerate}[\rm (i)]
\item $\kappa_{i}^{t}(A)$ is a polynomial in $p_{1}(A), \dots, p_{i}(A)$ with leading term $\frac{t^{i-1}}{(t)_{i}} p_{i}(A)$. 
That is, there exists a polynomial $F_{i}$ in $i-1$ variables such that 
$$
\kappa_{i}^t(A) = \frac{t^{i-1}}{(t)_{\underline{i}}}p_{i}(A) + F_{i}(p_{1}(A), \dots, p_{i-1}(A)).
$$ 
\item Homogeneity: $\kappa_{i}^{t}(D_{r}[A]) = r^{i}\kappa_{i}^{t}(A)$ for any $r \in \mathbb{R}$ and $i \in \mathbb{Z}_{\ge1}$.
\item Additivity: $\kappa_{i}^{t}(A \boxplus^{t} B) = \kappa_{i}^{t}(A) + \kappa_{i}^{t}(B)$. 
In other words, 
$$
C^t[A \boxplus^{t} B](z) = C^{t}[A](z) + C^{t}[B](z).
$$ 
\end{enumerate}
\end{theorem}

\begin{proof}
\begin{enumerate}[\rm (i)]
\item The definition $M[A](z) = -z\frac{d}{dz}\log(A(z))$ and \eqref{E^t} yield
$$
tC^{t}[A](z) = -z\frac{d}{dz}\log\bigl(E^{t}[A](z)\bigr).
$$ 
Consequently, we obtain
$$
M[A](z) \cdot A(z) = -z\frac{d}{dz}A(z) \quad \text{and} \quad tC^{t}[A](z) \cdot E^{t}[A](z) = -z\frac{d}{dz}E^{t}[A].
$$ 
Comparing the coefficients, for any $i \in \mathbb{Z}_{\ge1}$, we have 
\begin{align*}
-ia_{i} &= p_{i}(A) + p_{i-1}(A)a_{1} + \dots + p_{1}(A)a_{i-1} \\
t\kappa_{i}^{t}(A) &= -\frac{t^{i}}{(t)_{\underline{i}}}ia_{i} - t\kappa_{1}^{t}(A)\frac{t^{i-1}}{(t)_{\underline{i-1}}}a_{i-1} - \dots - t\kappa_{i-1}^{t}\frac{t^{1}}{(t)_{\underline{1}}}a_{1}. 
\end{align*}
One can show that $-ia_{i}$ and $t\kappa_{i}^{t}(A)$ are polynomials in $p_{1}(A), \dots, p_{i}(A)$ with leading terms $p_{i}(A)$ and $\frac{t^{i}}{(t)_{\underline{i}}}p_{i}(A)$ by induction on $i$. 

\item Since $E^{t}[D_{r}[A]](z) = D_{r}[E^{t}[A]](z)$ and 
\begin{align*}
tC^{t}[A](z) = -z\frac{d}{dz}\log \bigl(E^{t}[A](z)\bigr) = M[E^{t}[A]](z),
\end{align*}
one can verify that
\begin{align*}
tC^{t}[D_{r}[A]](z) 
&= M[E^{t}[D_{r}[A]]](z)\\
&= M[D_{r}[E^{t}[A]]] \\
&= M[E^{t}[A]](rz) = tC^{t}[A](rz) 
\end{align*}
and hence $\kappa_{i}^{t}(D_{r}[A]) = r^{i}\kappa_{i}^{t}(A)$ for any $i \in \mathbb{Z}_{\ge1}$. 

\item A direct computation shows that
\begin{align*}
C^t[A\boxplus^t B](z) 
&= -\frac{z}{t} \frac{d}{dz}\log\left(\sum_{k=0}^\infty \frac{t^k}{(t)_{\underline{k}}} \sum_{i+j=k} \frac{(t)_{\underline{k}}}{(t)_{\underline{i}}(t)_{\underline{j}}} a_{i}b_{j}z^k \right)\\
&=-\frac{z}{t} \frac{d}{dz}\log\left( \sum_{k=0}^\infty  \sum_{i+j=k}  \frac{t^i}{(t)_{\underline{i}}} a_{i} \cdot \frac{t^j}{(t)_{\underline{j}}} b_{j} z^k\right)\\
&=-\frac{z}{t} \frac{d}{dz}\log \left[ \left( \sum_{i=0}^\infty \frac{t^i}{(t)_{\underline{i}}} a_{i} z^i\right)\left( \sum_{j=0}^\infty \frac{t^j}{(t)_{\underline{j}}} b_{j} z^j\right)\right]\\
&=-\frac{z}{t} \frac{d}{dz}\log \bigl(E^{t}[A](z)\bigr) -\frac{z}{t} \frac{d}{dz}\log \bigl(E^{t}[B](z)\bigr)\\
&=C^t[A](z) +C^t[B](z).
\end{align*}
Therefore $\kappa_{i}^{t}(A\boxplus^t B)=\kappa_{i}^{t}(A)+\kappa_{i}^{t}(B)$ by comparing these coefficients.
\end{enumerate}

The uniqueness follows from the axioms and the homogeneity of $p_{i}$ (Lemma \ref{homogeneity}(1)) by the standard argument used in the proof of \cite[Theorem 3.1]{hasebe2011monotone-adlhppes}. 
\end{proof}

\begin{remark}\label{rem:d-deformed result}
\begin{enumerate}[\rm (1)]

\item For any $f\in \mathbb{R}[x]_{{\rm monic},d}$, we get
$$
C^d[\iota_d(f)](z) \equiv \sum_{i=1}^d \kappa_i^{(d)}(f) z^i \qquad \text{(mod $z^{d+1}$)}.
$$
Therefore we can obtain Theorem \ref{Arizmensi-Perales} from Theorem \ref{t-cumulants}.
Thus, the $t$-deformed cumulants provide a unified framework for the finite free cumulants. 

\item Arizmendi and Perales presented explicit formulas of relations between cumulants and moments. 
One can show that the same formulas hold for $t$-deformed cumulants by formally similar computation. 
See \cite[Proposition 3.4, Lemma 4.1, and Theorem 4.2]{arizmendi2018cumulants-joctsa} for the description. 
\end{enumerate}
\end{remark}

\subsection{Examples}\label{sec2.2}

In this subsection, we present several important examples of formal power series from a probabilistic perspective.
\begin{example}[Binomial series]
For $t\in \mathbb{R}^\times$ and $\lambda\in \mathbb{R}$, we define
\[
B_\lambda^{(t)}(z) \coloneqq (1-\lambda z)^t = 1+\sum_{k=1}^\infty (-1)^k \frac{(t)_{\underline{k}}}{k!} \lambda^k z^k.
\]
We call $B_\lambda^{(t)}$ the \textit{$t$-deformed binomial series}. For $t\in \mathbb{R}\setminus \mathbb{Z}_{\ge0}$, we obtain
\begin{align*}
C^t\bigl[B_\lambda^{(t)}\bigr](z) 
&= -\frac{z}{t} \frac{d}{dz}\log \left( 1+ \sum_{k=1}^\infty  \frac{t^k}{(t)_{\underline{k}}} \cdot (-1)^k \frac{(t)_{\underline{k}}}{k!} \lambda^k z^k\right)\\
&= -\frac{z}{t} \frac{d}{dz}\log \left( 1+ \sum_{k=1}^\infty  \frac{(-t\lambda z)^k}{k!} \right)\\
&=-\frac{z}{t} \frac{d}{dz}\log \bigl(\exp (-t\lambda z)\bigr)\\
&=\lambda z.
\end{align*}
 For $d\in \mathbb{Z}_{\ge 1}$, we have
 $$
 C^d\bigl[B_\lambda^{(d)}\bigr](z) = \lim_{t\to d}C^t\bigl[B_\lambda^{(t)}\bigr](z) =\lambda z.
 $$
Consequently, we obtain the following results:
\begin{itemize}
\item $\kappa_1^t\bigl(B_\lambda^{(t)}\bigr)=\lambda$ and $\kappa_n^t\bigl(B_\lambda^{(t)}\bigr)=0$ for all $n\ge 2$.
\item For all $\lambda_1,\lambda_2 \in \mathbb{R}$, we have
\begin{align}\label{eq:semigrp_binomial}
B_{\lambda_1}^{(t)}\boxplus^t B_{\lambda_2}^{(t)} = B_{\lambda_1+\lambda_2}^{(t)}
\end{align}
\item For $d\in \mathbb{Z}_{\ge 1}$, we get $B_\lambda^{(d)} =\iota_d(p) \in U_d$, where $p(x)=(x-\lambda)^d$.
\end{itemize}
\end{example}

\begin{example}[Hermite series]\label{ex:Hermite}
For $t\in\mathbb{R}^\times$, we define
\[
H^{(t)}(z) \coloneqq 1+\sum_{k=1}^\infty (-1)^k \frac{(t)_{\underline{2k}}}{t^k (2k)!!} z^{2k}.
\]
We call $H^{(t)}$ the \textit{$t$-deformed Hermite series}. 
For $t\in \mathbb{R}\setminus \mathbb{Z}_{\ge0}$, we get
\begin{align*}
C^t\bigl[H^{(t)}\bigr](z) 
&=-\frac{z}{t} \frac{d}{dz}\log \left(1+ \sum_{k=1}^\infty   \frac{t^{2k}}{(t)_{\underline{2k}}} \cdot (-1)^k \frac{(t)_{\underline{2k}}}{t^k(2k)!!} z^{2k}\right)\\
&=-\frac{z}{t} \frac{d}{dz}\log \left(1+ \sum_{k=1}^\infty  \frac{(-tz^2/2)^k}{k!}  \right)\\
&=-\frac{z}{t} \frac{d}{dz}\log \left( \exp\left( -\frac{tz^2}{2} \right)\right)\\
&=z^2.
\end{align*}
 For $d\in \mathbb{Z}_{\ge 1}$, we have
$$
C^d\bigl[H^{(d)}\bigr](z) =\lim_{t\to d}C^t\bigl[H^{(t)}\bigr](z)  =z^2. 
$$
Therefore we get the following properties:
\begin{itemize}
\item $\kappa_2^t\bigl(H^{(t)}\bigr)=1$ and $\kappa_n^t\bigl(H^{(t)}\bigr)=0$ for all $n\neq 2$. 
\item For $d\in \mathbb{Z}_{\ge 1}$, we have
$$
H^{(d)}(z) = \iota_{d}(H_d)(z)\in U_d,
$$
where $H_d(x) = \sum_{k=0}^{\lfloor d/2 \rfloor} (-1)^k \frac{(d)_{\underline{2k}}}{d^k(2k)!!}x^{d-2k}$ is the Hermite polynomial. 
\end{itemize}
\end{example}
Finally, we show that the family is recursive with respect to $\boxplus^t$, as follows.
\begin{corollary}\label{cor:recursive}
For $t\in \mathbb{R}^\times$, we define $H_s^{(t)}\coloneqq D_{\sqrt{s}}\bigl(H^{(t)}\bigr)$ for each $s\ge 0$. Then we obtain
$$
H_{s_1}^{(t)} \boxplus^t H_{s_2}^{(t)} = H_{s_1+s_2}^{(t)}, \qquad s_1,s_2\ge 0.
$$
\end{corollary}
\begin{proof}

This result is already known in the setting of finite free probability (see \cite[Page 813]{marcus2022finite-ptarf}), 
and the same argument extends naturally to our $t$-deformed framework. 
Since $C^t\bigl[H^{(t)}\bigr](z) = z^2$, we obtain
\[
C^t\bigl[H_s^{(t)}\bigr](z) = s z^2, \qquad s \ge 0.
\]
Therefore,
\[
C^t\bigl[H_{s_1}^{(t)} \boxplus^t H_{s_2}^{(t)}\bigr](z)
= (s_1 + s_2) z^2
= C^t\bigl[H_{s_1+s_2}^{(t)}\bigr](z),
\]
which completes the proof by Lemma~\ref{lem:injective}.
\end{proof}

\begin{example}[Laguerre series]\label{ex:Laguerre}

For $t\in \mathbb{R}^\times$ and $\lambda>0$, we define
\[
L_\lambda^{(t)}(z) \coloneqq 1+ \sum_{k=1}^\infty (-1)^k \frac{(\lambda t)_{\underline{k}}(t)_{\underline{k}}}{t^k k !}z^k.
\]
We call $L_\lambda^{(t)}$ the \textit{$t$-deformed Laguerre series}. For $t\in \mathbb{R}\setminus \mathbb{Z}_{\ge0}$, we get
\begin{align*}
C^t\bigl[L_\lambda^{(t)}]\bigr(z) 
&=-\frac{z}{t} \frac{d}{dz} \log \left( 1+\sum_{k=1}^\infty (-1)^k \frac{(\lambda t)_{\underline{k}}}{k!}z^k\right)\\
&=-\frac{z}{t} \frac{d}{dz} \log \bigl(  (1-z)^{\lambda t}\bigr)\\
&=\frac{\lambda z}{1-z}.
\end{align*}
For $d \in \mathbb{Z}_{\ge1}$, we obtain
$$
C^d\bigl[L_\lambda^{(d)}\bigr](z) = \lim_{t\to d} C^t\bigl[L_\lambda^{(t)}\bigr](z)  = \frac{\lambda z}{1-z}.
$$
Consequently, we obtain the following results:
\begin{itemize}
\item $\kappa_n^t\bigl(L_\lambda^{(t)}\bigr)=\lambda$ for all $n\ge 1$.
\item For $\lambda_1, \lambda_2>0$, we have
\begin{align}\label{eq:Lag_convolution}
L_{\lambda_1}^{(t)} \boxplus^t L_{\lambda_2}^{(t)}  = L_{\lambda_1+\lambda_2}^{(t)}. 
\end{align}
\item For $d\in \mathbb{Z}_{\ge1}$, we get 
$$
L_\lambda^{(d)}(z) = \iota_d(L_{d,\lambda})\in U_d,
$$ 
where $L_{d,\lambda}(x)=\sum_{k=0}^d (-1)^k \frac{(\lambda d)_{\underline{k}}(d)_{\underline{k}}}{d^k k !}x^{d-k}$ is the Laguerre polynomial.
\end{itemize}
\end{example}

\subsection{Hypergeometric series}

The class of hypergeometric polynomials contains many important (real-rooted) polynomials, such as the Laguerre, Jacobi, and Bessel polynomials, and exhibits a rich structure. Recently, these polynomials have been studied in connection with finite free probability in \cite{campbell2025even, martinez2024real, martinez2025zeros}.

Recall the hypergeometric polynomials as follows. 
Let $i, j \in \mathbb{Z}_{\ge 0}$ and $d \in \mathbb{Z}_{\ge 1}$. 
Let $a_1, \dots, a_i \in \mathbb{R}_{|d} \coloneqq \mathbb{R} \setminus \left\{0, \frac{1}{d}, \dots, \frac{d-1}{d}\right\}$ and $b_1, \dots, b_j \in \mathbb{R}$. 
We define the {\it hypergeometric polynomial} as
\begin{align}\label{HGP}
\mathcal{H}_d \begin{bmatrix} b_1,\dots, b_j\\
a_1,\dots, a_i
\end{bmatrix}(x)\coloneqq  \sum_{k=0}^d(-1)^k \binom{d}{k} \frac{(db_1)_{\underline{k}}\cdots (db_j)_{\underline{k}}}{(da_1)_{\underline{k}}\cdots (da_i)_{\underline{k}}}x^{d-k}.
\end{align}
To simplify notation, for $\bm{a}=(a_1,\dots, a_i)$, we write
$$
(\bm{a})_{\underline{k}}\coloneqq \prod_{r=1}^i (a_r)_{\underline{k}}, \qquad (\bm{a})_{\bar{k}}\coloneqq \prod_{r=1}^i (a_r)_{\bar{k}},
$$
where $(a)_{\bar{k}}\coloneqq a(a+1)\cdots (a+k-1)$ for $a\in \mathbb{R}$. Then, for $\bm{a}=(a_1,\dots, a_i) \in (\mathbb{R}_{|d})^i$ and $\bm{b}=(b_1,\dots, b_j)\in \mathbb{R}^j$, the hypergeometric polynomial \eqref{HGP} is written by
$$
\mathcal{H}_d\begin{bmatrix} \bm{b}\\
\bm{a}
\end{bmatrix}(x) =\sum_{k=0}^d(-1)^k \binom{d}{k} \frac{(d\bm{b})_{\underline{k}}}{(d\bm{a})_{\underline{k}}}x^{d-k}.
$$

These polynomials are closely related to generalized hypergeometric series. More precisely,
$$
\mathcal{H}_d\begin{bmatrix} \bm{b}\\
\bm{a}
\end{bmatrix}(x)  = \frac{(-1)^d(d\bm{b})_{\underline{d}}}{(d\bm{a})_{\underline{d}}}
 \ {}_{i+1}F_j\left(\begin{array}{c}
-d, \ d\bm{a}-d+1 \\
d\bm{b}-d+1
\end{array};\, x\right),
$$
where $d\bm{a}-d+1 \coloneqq (da_1-d+1,\dots da_i-d+1)$ and for $\bm{c} \in \mathbb{R}^n$ and $\bm{d}\in \mathbb{R}^m$, we define
$$
{}_{m}F_n\left(\begin{array}{c}
\bm{d}\\
\bm{c}
\end{array};\, x\right) \coloneqq  \sum_{k=0}^\infty \frac{1}{k!} \frac{(\bm{d})_{\bar{k}}}{(\bm{c})_{\bar{k}}}  x^k,
$$
see \cite{koekoek2010hypergeometric} for details.

We extend hypergeometric polynomials to formal power series. Let $d\in \mathbb{Z}_{\ge1}$. We consider the image of a hypergeometric polynomial under the map $\iota_d$ introduced in Section \ref{sec2}. For $\bm{a} \in (\mathbb{R}_{|d})^i$ and $\bm{b}\in \mathbb{R}^j$, we have
\begin{align*}
\iota_d \left( \mathcal{H}_d\begin{bmatrix}
\bm{b}\\
\bm{a}
\end{bmatrix} \right)(z) = \sum_{k=0}^d(-1)^k \frac{(d)_{\underline{k}}}{k!} \frac{(d\bm{b})_{\underline{k}}}{(d\bm{a})_{\underline{k}}}z^k \in U_d.
\end{align*}
Note that $a \in \mathbb{R}_{|d}$ if and only if $da \notin \{0,1,\dots , d-1\}$. Motivated by this observation, we define an extension of hypergeometric polynomials as follows.
\begin{definition}
For a given $t\in \mathbb{R}^\times$, we consider $\bm{a}\in \mathbb{R}^i$ such that $ta_k \notin \mathbb{Z}_{\ge0}$ for all $k=1,\dots, i$ and $\bm{b}\in\mathbb{R}^j$. The {\it $t$-deformed Hypergeometric series} is defined by
$$
\mathcal{H}^{(t)} \begin{bmatrix}
\bm{b}\\
\bm{a}
\end{bmatrix} (z)\coloneqq   \sum_{k=0}^\infty(-1)^k \frac{(t)_{\underline{k}}}{k!} \frac{(t\bm{b})_{\underline{k}}}{(t\bm{a})_{\underline{k}}}z^k.
$$
\end{definition}

One can easily translate $t$-deformed hypergeometric series into generalized hypergeometric series. Since $(-\bm{a})_{\underline{k}} = (-1)^{ki} (\bm{a})_{\bar{k}}$ for $\bm{a} \in \mathbb{R}^i$, we have
\begin{align*}
\mathcal{H}^{(t)} \begin{bmatrix}
\bm{b}\\
\bm{a}
\end{bmatrix} (z) = \ {}_{i+1}F_j\left(\begin{array}{c}
-t, \ -t\bm{a} \\
-t\bm{b}
\end{array};\, (-1)^{j-i}z \right).
\end{align*}

\begin{example}
\begin{enumerate}[\rm (1)]
\item ($t$-deformed binomial series) For $\lambda\in \mathbb{R}$, we have
$$
D_\lambda\left[\mathcal{H}^{(t)} \begin{bmatrix}
-\\
-
\end{bmatrix} \right](z) = B_\lambda^{(t)}(z).
$$

\item ($t$-deformed Laguerre series) For $b>0$, we have
$$D_{1/t}\left[\mathcal{H}^{(t)} \begin{bmatrix}
b\\
-
\end{bmatrix} \right](z) = L_b^{(t)}(z).
$$

\item ($t$-deformed Bessel series) For $a\in \mathbb{R}$ with $ta \notin\mathbb{Z}_{\ge0}$, 
$$
\displaystyle D_{t}\left[\mathcal{H}^{(t)} \begin{bmatrix}
-\\
a
\end{bmatrix} \right](z) = \sum_{k=0}^\infty (-1)^k \frac{(t)_{\underline{k}}}{k!}\frac{t^k}{(ta)_{\underline{k}}} z^k.
$$

\item ($t$-deformed Jacobi series) For $a\in \mathbb{R}$ with $ta\notin\mathbb{Z}_{\ge0}$ and $b>0$,
$$
\mathcal{H}^{(t)} \begin{bmatrix}
b\\
a
\end{bmatrix} (z) =\sum_{k=0}^\infty(-1)^k \frac{(t)_{\underline{k}}}{k!} \frac{(tb)_{\underline{k}}}{(ta)_{\underline{k}}} z^k.
$$
\end{enumerate}
\end{example}

Finally, we show that $t$-deformed hypergeometric series are closed under $\boxplus^t$ subject to a certain condition. This result can be regarded as a $t$-analogue of \cite[(83),(84)]{martinez2024real}.
\begin{proposition}\label{prop:t-convolution:HGFPS}
Let us consider $t\in \mathbb{R}^\times$. If we consider $\bm{a}_\ell$ and $\bm{b}_\ell$ with lengths $i_\ell$ and $j_\ell$ ($\ell=1,2,3$), respectively. Then 
\begin{align}\label{HGSeq}
{}_{j_1} F_{i_1}\left(\begin{array}{c}
-t\bm{b}_1\\
-t \bm{a}_1
\end{array};\, x\right) 
{}_{j_2} F_{i_2}\left(\begin{array}{c}
-t\bm{b}_2\\
-t \bm{a}_2
\end{array};\, x\right)
=
{}_{j_3} F_{i_3}\left(\begin{array}{c}
-t\bm{b}_3\\
-t \bm{a}_3
\end{array};\, x\right)
\end{align}
if and only if
\begin{align}\label{HGFPS}
\mathcal{H}^{(t)}\begin{bmatrix}
\bm{b}_1\\
\bm{a}_1
\end{bmatrix} (s_1z)
\boxplus^t
\mathcal{H}^{(t)}\begin{bmatrix}
\bm{b}_2\\
\bm{a}_2
\end{bmatrix} (s_2z)
=
\mathcal{H}^{(t)}\begin{bmatrix}
\bm{b}_3\\
\bm{a}_3
\end{bmatrix} (s_3z),
\end{align}
where $s_\ell = (-1)^{i_\ell+ j_\ell +1}$ for $\ell =1,2,3$.
\end{proposition}

\subsection{Limit theorems for $t$-deformed convolution}\label{sec2.3}

In this subsection, we establish the law of large numbers and the central limit theorem for the $t$-deformed convolution. We define the \textit{$t$-deformed convolution power} of $A$ by
\[
A^{\boxplus^t m} \coloneqq \underbrace{A\boxplus^t \cdots \boxplus^t A}_{m \text{ times}}.
\]
When considering the limit of a sequence $\{A_n(z)\}_{n\in\mathbb N}$ of formal power series, we equip $\mathbb{R}\llbracket z \rrbracket$ with the topology induced from the product topology on $\prod_{n=0}^{\infty}\mathbb{R}$, to which it is naturally homeomorphic.
Then $\{A_n(z)\}_{n\in\mathbb N}$ converges if and only if every coefficient of $A_{n}(z)$ converges. Moreover, this is equivalent to that every $t$-deformed cumulants of $A_n(z)$ converges.
Note that this topology is finer than the $(z)$-adic topology of $\mathbb{R}\llbracket z \rrbracket$. 

\begin{theorem}\label{limit_thms}
\begin{enumerate}[\rm (1)]
\item (\textit{Law of large numbers for $\boxplus^t$})\ For $t\in \mathbb{R}^\times$, let us consider $A\in U$ such that $\kappa_1^t(A)=\lambda$. Then
\begin{align*}
D_{1/m}\bigl[A^{\boxplus^t m}\bigr] (z) \to B_\lambda^{(t)}(z), \qquad m\to \infty.
\end{align*}
\item (\textit{Central limit theorem for $\boxplus^t$})\ For $t\in \mathbb{R}^\times$, let us consider $A\in U$ such that $\kappa_1^t(A)=0$ and $\kappa_2^t(A)=1$. Then
\begin{align*}
D_{1/\sqrt{m}} \bigl[A^{\boxplus^t m} \bigr] (z)  \to H^{(t)}(z), \qquad m \to\infty.
\end{align*} 
\end{enumerate}
\end{theorem}
\begin{proof}
\begin{enumerate}[\rm (1)]
\item By Theorem \ref{t-cumulants} (2),(3), for each $n\in \mathbb{Z}_{\ge1}$, we have
\begin{align*}
\kappa_n^t\bigl(D_{1/m}\bigl[A^{\boxplus^t m}\bigr]\bigr) =  \frac{1}{m^{n-1}} \kappa_n^t(A).
\end{align*}
Therefore
\begin{align*}
\lim_{m\to \infty} C^t\bigl[D_{1/m}\bigl[A^{\boxplus^t m}\bigr]\bigr](z) =\lambda z = C^t\bigl[B_\lambda^{(t)}\bigr](z).
\end{align*}
Finally, we have $\lim_{m\to \infty}D_{1/m}\bigl[A^{\boxplus^t m}\bigr](z) = B_\lambda^{(t)}(z)$ by Lemma \ref{lem:injective}.

\item By Theorem \ref{t-cumulants} (2),(3), for each $n\in \mathbb{Z}_{\ge1}$, we have
\begin{align*}
\kappa_n^t\bigl(D_{1/\sqrt{m}}\bigl[A^{\boxplus^t m}\bigr]\bigr) =  \frac{1}{m^{n/2-1}} \kappa_n^t(A).
\end{align*}
Therefore
\begin{align*}
\lim_{m\to \infty} C^t\bigl[D_{1/m}\bigl[A^{\boxplus^t m}\bigr]\bigr](z) =z^2 = C^t\bigl[H^{(t)}\bigr](z).
\end{align*}
Finally, we have $\lim_{m\to \infty}D_{1/\sqrt{m}}\bigl[A^{\boxplus^t m}\bigr](z) = H^{(t)}(z)$ by Lemma \ref{lem:injective}.
\end{enumerate}
\end{proof}

\section{Applications to probability theory}\label{sec:3}

In this section, we focus on the case $t=-1$ to apply our theorems to probability theory. In this case, the $t$-deformed convolution is given by 
\begin{align}\label{-1-convolution}
(A \boxplus^{-1} B)(z) = \sum_{k=0}^{\infty}\sum_{i+j=k}\binom{k}{i}a_{i}b_{j}z^{k}.
\end{align}
Under its specialization, several important formal power series admit natural interpretations from classical probabilistic viewpoints. 

\subsection{Classical cumulant transform}\label{sec3.1}

For a real-valued random variable $X$ on some probability space, we define the ordinary and exponential generating functions of $X$ as
\[
M_X(z) \coloneqq \sum_{k=0}^\infty  \mathbb{E}[X^k] z^k \quad \text{and} \quad \Psi_X(z) \coloneqq \sum_{k=0}^\infty \frac{\mathbb{E}[X^k]}{k!} z^k,
\]
respectively. Moreover, the \textit{classical cumulant transform} of $X$ is defined by
\[
C_X^\ast(z) \coloneqq \log\bigl(\Psi_X({\rm i} z)\bigr).
\]
It is well-known that $C_X^\ast(z)$ behaves well under the summation of independent random variables. 
More precisely, if $X$ and $Y$ are independent real-valued random variables, 
then 
$$
C_{X+Y}^\ast(z) = C_{X}^\ast(z) + C_{Y}^\ast(z).
$$ 
In particular, if $X$ has all finite moments, then there exists a unique sequence $\{r_n(X)\}_n \subset \mathbb{R}$ such that
\[
C_X^\ast(z) =\sum_{n=1}^\infty r_n(X)z^n
\]
as formal power series. The value $r_n(X)$ is called the $n$-th \textit{cumulant} of $X$.

We will show that $M_{X}(z)$ behaves well under the convolution $\boxplus^{-1}$ and that the $(-1)$-deformed cumulant transform is closely related to the classical cumulant transform. 
\begin{theorem}\label{Classicalresult}
The following statements hold.
\begin{enumerate}[\rm (1)]
\item If a real-valued random variable $X$ and $Y$ are independent, then $M_{X+Y} = M_{X} \boxplus^{-1} M_{Y}$. 
\item $\displaystyle C^{-1}[M_X]({\rm i} z) =  z\frac{d}{dz}C_X^\ast(z)$.
\item If $X$ has all finite moments, then $nr_n(X) = {\rm i}^n \kappa_n^{-1}(M_X)$ for all $n\in \mathbb{Z}_{\ge1}$.
\end{enumerate}
\end{theorem}
\begin{proof}
\begin{enumerate}[\rm (1)]
\item By \eqref{-1-convolution}, we have
\begin{align*}
(M_X\boxplus^{-1}M_Y)(z) 
&= \sum_{k=0}^\infty \sum_{i+j=k} \binom{k}{i}\mathbb{E}[X^i] \mathbb{E}[Y^j]  z^k\\
&=\sum_{k=0}^\infty \mathbb{E}\left[\sum_{i+j=k} \binom{k}{i} X^i Y^j\right] z^k \\
&=\sum_{k=0}^\infty \mathbb{E}[(X+Y)^k] z^k,
\end{align*}
where the second equality holds since $X$ and $Y$ are independent.
Therefore we obtain $(M_X\boxplus^{-1}M_Y)(z)=M_{X+Y}(z)$.

\item A direct computation shows
\begin{align*}
E^{-1}[M_X](z) 
&= 1+ \sum_{k=1}^\infty  \frac{(-1)^k}{(-1)_{\underline{k}}} \mathbb{E}[X^k]z^k \\
&= 1+ \sum_{k=1}^\infty \frac{\mathbb{E}[X^k]}{k!} z^k  =\Psi_X(z).
\end{align*}
Consequently, we obtain
\begin{align*}
C^{-1}[M_X]({\rm i}z) 
&= z\frac{d}{dz} \log \bigl(E^{-1}[M_X]({\rm i}z) \bigr)\\
&= z\frac{d}{dz} \log \bigl(\Psi_X({\rm i}z)\bigr) = z\frac{d}{dz}C_X^\ast(z).
\end{align*}

\item By the above formula, we obtain
$$
\sum_{n=1}^\infty {\rm i}^n \kappa_n^{-1}(M_X)z^n = \sum_{n=1}^\infty n r_n(X) z^n.
$$
By comparing coefficients, we conclude that $ {\rm i}^n \kappa_n^{-1}(M_X)=n r_n(X)$ for all $n\in \mathbb{Z}_{\ge1}$.
\end{enumerate}
\end{proof}

\subsection{Mixture of beta and gamma distributions}\label{sec3.2}

In this section, we focus on the case $t = -1$ for $t$-deformed hypergeometric series. In this case, these formal power series are related to mixtures of certain beta and gamma distributions. Recall that the {\it beta distribution} with parameter $(\alpha,\beta)\in \mathbb{R}_{>0}^2$, denoted by $\text{Beta}(\alpha, \beta)$, is the probability measure on $\mathbb{R}$ whose probability density function is given by
$$
\frac{1}{B(\alpha,\beta)} x^{\alpha-1}(1-x)^{\beta-1} \mathbf{1}_{[0,1]}(x),
$$
where $B(\alpha,\beta)$ is the beta function. The {\it gamma distribution} with parameter $\alpha>0$, denoted by $\text{Ga}(\alpha)$, is the probability measure on $\mathbb{R}$ whose probability density function is given by
$$
\frac{1}{\Gamma(\alpha)} x^{\alpha-1} e^{-x} \mathbf{1}_{[0,\infty)}(x),
$$
where $\Gamma(\alpha)$ is the gamma function.

In what follows, we make use of the relation
$$
(a)_{\overline{k}} = \frac{\Gamma(a+k)}{\Gamma(a)}, \qquad a>0, \ k\ge 1.
$$

\begin{proposition}\label{thm:Meijer}
Consider $j \ge i \ge 1$ with $j-i \in 2\mathbb{Z}_{\ge0}$. Let $\bm{a}=(a_1,\dots, a_i) \in \mathbb{R}_{>0}^i$ and $\bm{b}=(b_1,\dots, b_j) \in \mathbb{R}_{>0}^j$ be tuples such that $a_r-b_r>0$ for all $1\le r \le i$. Then we obtain
$$
\mathcal{H}^{(-1)} \begin{bmatrix}
\bm{b}\\
\bm{a}
\end{bmatrix} (z) = M_X(z)
$$
where $X$ is a random variable distributed as 
$$
\text{\rm Beta}(b_1, a_1-b_1) \circledast \cdots \circledast \text{\rm Beta}(b_i, a_i-b_i) \circledast \text{\rm Ga}(b_{i+1})\circledast \dots \circledast \text{\rm Ga}(b_j),
$$
where $\mu \circledast \nu$ denotes the classical multiplicative convolution, that is, the distribution of the product $XY$ of two independent positive random variables $X \sim \mu$ and $Y \sim \nu$.
\end{proposition}
\begin{proof}
Notice that, for all $k \in \mathbb{Z}_{\ge1}$,
$$
\mathbb{E}[X^k] = \frac{\prod_{s=1}^j \Gamma(b_s +k)}{\prod_{s=1}^j \Gamma(b_s)} \cdot \frac{\prod_{r=1}^i \Gamma(a_r)}{\prod_{r=1}^i \Gamma(a_r+k)}.
$$
On the other hand, we have
\begin{align*}
\mathcal{H}^{(-1)} \begin{bmatrix}
\bm{b}\\
\bm{a}
\end{bmatrix} (z) 
&=\sum_{k=0}^\infty (-1)^k \frac{(-1)_{\underline{k}}}{k!} \frac{(- \bm{b})_{\underline{k}}}{(-\bm{a})_{\underline{k}}}z^k\\
&=\sum_{k=0}^\infty \frac{(-1)^{kj} (\bm{b})_{\bar{k}}}{(-1)^{ki} (\bm{a})_{\bar{k}}}z^k\\
&=\sum_{k=0}^\infty (-1)^{k(j-i)} \frac{\prod_{s=1}^j (b_s)_{\bar{k}}}{\prod_{r=1}^i (a_r)_{\bar{k}}}z^k\\
&=\sum_{k=0}^\infty  \frac{\prod_{s=1}^j\Gamma(b_s +k)}{\prod_{s=1}^j\Gamma(b_s)} \frac{\prod_{r=1}^i\Gamma(a_r)}{\prod_{r=1}^i\Gamma(a_r+k)}z^k\\
&=\sum_{k=0}^\infty \mathbb{E}[X^k]z^k = M_X(z).
\end{align*}
\end{proof}

The $t$-deformed Laguerre series can be understood as the ordinary generating function of gamma random variables. More precisely, if $b>0$, then
\begin{align}\label{Laguerre-Gamma}
L_b^{(-1)}(z) = \sum_{k=0}^\infty  (b)_{\overline{k}}z^k = M_X(z),
\end{align}
where a random variable $X$ is distributed as ${\rm Ga}(b)$.

\subsection{Self-decomposable random variables}\label{sec3.3}

A real-valued random variable $X$ is said to be \textit{self-decomposable} if for any $c\in (0,1)$ there exists a random variable $X_c$ such that $X_c$ and $X$ are independent and $X \overset{d}{=} cX+X_c$, where $\overset{d}{=}$ denotes equality in distribution. That is, if we denote by $\mu$ the law of self-decomposable random variable $X$, then for any $c \in (0,1)$, there exists a probability measure $\mu_c$ on $\mathbb{R}$ such that
$$
\mu = D_c(\mu) \ast \mu_c,
$$
where $D_c(\mu)$ is the dilation by $c$, namely $D_c\mu(B)=\mu(c^{-1}B)$ for all Borel sets $B$ in $\mathbb{R}$.
A distribution $\mu$ of a self-decomposable random variable is known to be {\it infinitely divisible}, that is, for any $n\in\mathbb{N}$, there exists a (unique) probability measure $\mu_n$ on $\mathbb{R}$ such that 
$$
\mu= \underbrace{\mu_n\ast \cdots \ast \mu_n}_{n \text{ times}}.
$$
A one-dimensional real-valued stochastic process $\{X(s)\}_{s\ge0}$ on some probability space $(\Omega,\mathcal{F},P)$ is called a (one-dimensional) {\it L\'{e}vy process} if it satisfies the following properties:
\begin{enumerate}[\rm (i)]
\item $X(0)=0$ a.s.;
\item ({\it independent increments}) for any $0\le t_0<t_1<\cdots<t_n$, the random variables $X(t_1)-X(t_0), X(t_2)-X(t_1),\dots, X(t_n)-X(t_{n-1})$ are independent;
\item ({\it stationary increments}) for any $s,t>0$, the increment $X(t+s) - X(s)$ has the same distribution as $X(t)$;
\item ({\it stochastic continuity}) for any $\epsilon>0$ and $t\ge 0$,
$$
\lim_{s\to t} P\bigl(|X(s)-X(t)|>\epsilon \bigr) = 0;
$$
\end{enumerate}
In this case, it is known that $\{X(s)\}_{s\ge0}$ admits a c\'{a}dl\'{a}g modification. Hence, without loss of generality, we may assume that it is c\'{a}dl\'{a}g (i.e. right-continuous with left limits). It is well known that for each $t \ge 0$, the distribution of $X(t)$ is infinitely divisible; see e.g. \cite{ken1999levy} for details.

According to \cite[Theorem 3.2]{jurek1983integral}, there exists a one-dimensional L\'{e}vy process $\{Z(s)\}_{s\ge0}$ such that $\mathbb{E}[1+\log|Z(1)|]<\infty$ and
$$
X \overset{d}{=} \int_0^\infty e^{-s} dZ(s) ,
$$
where the right hand side is defined as the stochastic integral with respect to the L\'{e}vy process $\{Z(s)\}_{s\ge0}$. The L\'{e}vy process $\{Z(s)\}_{s\ge0}$ is called the \textit{background driving L\'{e}vy process} with respect to $X$. According to \cite[Page 6]{jurek2001remarks}, a useful formula to understand processes $\{Z(s)\}_{s\ge0}$ is known:
$$
C_{Z(s)}^\ast(z) =s z \frac{d}{dz} C_X^\ast(z), \qquad s\ge 0.
$$
Consequently, we have the following translation by using Theorem \ref{Classicalresult} (2).

\begin{proposition}\label{thm:BGL}
Let $X$ be a self-decomposable random variable with a background driving L\'{e}vy process $\{Z(s)\}_{s\ge0}$. For any $s\ge0$, we have
$$
C_{Z(s)}^\ast(z) = sC^{-1}[M_X]({\rm i}z).
$$
\end{proposition}

\begin{example}
\begin{enumerate}[\rm (1)]
\item Let $X$ be a random variable distributed as the standard normal distribution $N(0,1)$. One can verify that $\mathbb{E}[X^{2k}] = (2k-1)!!$, and therefore
\begin{align*}
H^{(-1)}(z) = \sum_{k=0}^\infty (2k-1)!! z^{2k} = M_X(z),
\end{align*}
where we understand $(2k-1)!!=1$ when $k=0$. Moreover, it is known that $X$ is self-decomposable and its background driving L\'{e}vy process $\{Z(s)\}_{s\ge0}$ is a one-dimensional Brownian motion $\{B(s)\}_{s\ge0}$, that is, a one-dimensional L\'{e}vy process such that $Z(1)$ follows the standard normal distribution $N(0,1)$. By Theorem \ref{thm:BGL}, we have
$$
C_{B(s)}^\ast (z) = s C^{-1}\bigl[H^{(-1)}\bigr]({\rm i}z), \qquad s\ge0.
$$

\item Let $X$ be a random variable distributed as the gamma distribution $\text{Ga}(\lambda)$. 
Then $X$ is self-decomposable. It is known that its background driving L\'{e}vy process $\{Z(s)\}_{s\ge0}$ is a compound Poisson process $\{G(s)\}_{s\ge0}$ with intensity $\lambda$ and standard exponential jumps; see \cite[Lemma 1]{jurek2022background}. Thus, \eqref{Laguerre-Gamma} and Theorem \ref{thm:BGL} imply that
$$
C_{G(s)}^\ast (z) = s C^{-1}\bigl[L_\lambda^{(-1)}\bigr]({\rm i}z), \qquad s\ge0.
$$
\end{enumerate}
\end{example}

\section{$(t,r)$-deformed Infinitesimal generators}\label{sec:4}

In this section, we introduce Banach algebras constructed from formal power series by imposing analytic growth conditions on their coefficients. We then define the $t$-deformed convolution operator on this space and study their infinitesimal generators.

\subsection{Banach algebra $(\mathcal{A}_r^{(t)},\boxplus^t)$}

For $r>0$ and $t\in\mathbb{R}\setminus \mathbb{Z}_{\ge0}$, we define
\begin{align*}
\mathcal{A}_{r}^{(t)}\coloneqq \left\{ \sum_{k=0}^\infty a_k z^k \in  \mathbb{R}\llbracket z \rrbracket: \sum_{k=0}^\infty \left| \frac{a_k}{(t)_{\underline{k}}}\right|r^k <\infty\right\}.
\end{align*}
Then $(\mathcal{A}_r^{(t)}, \boxplus^t)$ is a unital commutative algebra with unit element $1$. We also define
$$
\mathcal{A}_r\coloneqq \left\{ \sum_{k=0}^\infty a_k z^k\in \mathbb{R}\llbracket z \rrbracket : \sum_{k=0}^\infty |a_k|r^k <\infty \right\}.
$$
It is well known that $(\mathcal{A}_r, \cdot)$ is a (unital) Banach algebra over $\mathbb{R}$ with norm
$$
\|A\|_r := \sum_{k=0}^\infty |a_k|r^k.
$$

\begin{proposition}
For $r>0$ and $t\in\mathbb{R}\setminus \mathbb{Z}_{\ge0}$, the space $(\mathcal{A}_{r}^{(t)},\boxplus^t)$ is a Banach algebra over $\mathbb{R}$ with norm
$$
\|A\|_r^{(t)} \coloneqq  \sum_{k=0}^\infty \left| \frac{a_k}{(t)_{\underline{k}}}\right|r^k
\quad \text{for} \quad A(z)=\sum_{k=0}^\infty a_kz^k \in\mathcal{A}_{r}^{(t)},
$$
and is isometrically isomorphic to $(\mathcal{A}_r, \cdot)$.
\end{proposition}
\begin{proof}
We define an isomorphism $\Phi_t$ between $(\mathcal{A}_{r}^{(t)},\boxplus^t)$ and $(\mathcal{A}_r, \cdot)$ by
$$
\Phi_t \left( \sum_{k=0}^\infty a_k z^k\right)\coloneqq  \sum_{k=0}^\infty \frac{a_k}{(t)_{\underline{k}}}z^k.
$$
Clearly, $\Phi_t$ is isometry and
$$
\Phi_t(A\boxplus^t B) = \Phi_t(A) \Phi_t(B), \qquad A,B\in \mathcal{A}_{r}^{(t)}.
$$
Hence
\begin{align*}
\|A\boxplus^t B\|_{r}^{(t)} 
&= \|\Phi_t(A\boxplus^t B)\|_{r}=\|\Phi_t(A) \Phi_t( B)\|_{r}\\
&\le \|\Phi_t(A)\|_{r} \|\Phi_t(B)\|_{r} = \|A\|_{r}^{(t)}  \|B\|_{r}^{(t)}.
\end{align*}
Moreover, $\mathcal{A}_{r}^{(t)}$ is complete with respect to the norm $\|\cdot\|_{r}^{(t)}$, since the isometry $\Phi_t$ preserves completeness. Finally, $(\mathcal{A}_{r}^{(t)},\boxplus^t)$  is a Banach algebra with norm $\|\cdot\|_{r}^{(t)}$.
\end{proof}

\begin{remark}
If $A(z) = \sum_{k=0}^\infty a_k z^k \in \mathcal{A}_r^{(t)}$, then the series $\Phi_t(A)(z)$ converges absolutely for all $z\in \mathbb{D}_r\coloneqq \{z: |z|< r \}$.
\end{remark}

\subsection{$(t,r)$-deformed infinitesimal generators}

For $t\in\mathbb{R}\setminus \mathbb{Z}_{\ge0}$ and $r>0$, we define the \textit{$(t,r)$-deformed convolution operator} by
$$
T_B: \mathcal{A}_{r}^{(t)} \to \mathcal{A}_{r}^{(t)}, \qquad T_BA\coloneqq A\boxplus^t B.
$$
for a fixed element $B\in\mathcal{A}_{r}^{(t)}$. The case $t = d \in \mathbb{Z}_{\ge 1}$ will be discussed in Section~\ref{sec:4.3}.

Using the isometric isomorphism $\Phi_t: (\mathcal{A}_{r}^{(t)},\boxplus^t) \cong (\mathcal{A}_r,\cdot)$, we obtain
$$
\|T_B\| = \|B\|_{r}^{(t)}.
$$
Therefore, $T_B$ is a bounded operator. Moreover, we have
$$
T_A T_B = T_B  T_A = T_{A\boxplus^t B}
$$
for all $A,B \in \mathcal{A}_{r}^{(t)}$. Thus, $\{T_B: B\in \mathcal{A}_{r}^{(t)}\}$ forms a commutative semigroup.  

\begin{lemma}\label{lem:SCS}
If $\{Q_s\}_{s\ge 0} \subset  \mathcal{A}_{r}^{(t)}$ satisfies that 
\begin{itemize}
\item $Q_0=1$;
\item $Q_{s_1} \boxplus^t Q_{s_2} = Q_{s_1+s_2}$ for all $s_1,s_2\ge0$;
\item $\|Q_s- 1\|_{r}^{(t)} \to 0$ as $s\to 0^+$.
\end{itemize}
then $\{T_{Q_s}\}_{s\ge0}$ is a strongly continuous semigroup.
\end{lemma}
\begin{proof}
Let us write $T_s\coloneqq T_{Q_s}$ for each $s\ge0$. Clearly, $T_0 = I_{\mathcal{A}_r^{(t)}}$. For $s_1,s_2\ge0$, we have
$$
T_{s_1} T_{s_2}=T_{Q_{s_1}\boxplus^t Q_{s_2}}= T_{Q_{s_1+s_2}}=T_{s_1+s_2}.
$$
Moreover, for all $A\in \mathcal{A}_r^{(t)}$, we get
\begin{align*}
\|T_s A- A\|_r^{(t)} &=\|\Phi_t(T_sA-A)\|_{r} =\|\Phi_t(Q_s-1) \Phi_t(A)\|_{r}\\
&\le \|\Phi_t(A)\|_{r} \|\Phi_t(Q_s-1)\|_{r} = \|A\|_r^{(t)} \|Q_s-1\|_r^{(t)}\to 0
\end{align*}
as $s\to 0^+$.
\end{proof}

We call a family $\mathfrak{Q} = \{Q_s\}_{s \ge 0}$ satisfying the conditions of Lemma~\ref{lem:SCS} a \textit{$\boxplus^t$-semigroup}. We define its infinitesimal generator as follows.

\begin{definition}\label{def:generators}
Let $\mathfrak{Q}=\{Q_s\}_{s\ge0} \subset \mathcal{A}_r^{(t)}$ be a $\boxplus^t$-semigroup. Its \textit{$(t,r)$-deformed infinitesimal generator} $\mathcal{L}_r^{(t)}[\mathfrak{Q}]$ is defined by
$$
\mathcal{L}_r^{(t)}[\mathfrak{Q}] A \coloneqq  \lim_{s\to 0^+} \frac{T_s A -A}{s}, \qquad A\in \mathcal{D}\bigl(\mathcal{L}_r^{(t)}[\mathfrak{Q}] \bigr)
$$
with a domain
$$
\text{Dom}\bigl(\mathcal{L}_r^{(t)}[\mathfrak{Q}] \bigr)\coloneqq \left\{A\in \mathcal{A}_{r}^{(t)} : \lim_{s\to0^+} \frac{T_s A- A}{s} \text{ exists}\right\}.
$$
\end{definition}

\begin{theorem}\label{thm:generators}
Let $\mathfrak{Q}=\{Q_s\}_{s\ge 0} \subset \mathcal{A}_r^{(t)}$ be a $\boxplus^t$-semigroup. Assume that there exists the following limit:
$$
\eta_\mathfrak{Q}(z)\coloneqq \lim_{s\to 0^+} \frac{\Phi_t(Q_s)(z)-1}{s}, \quad \text{for each } z\in \mathbb{D}_r.
$$
Then we have
$$
\mathcal{L}_r^{(t)}[\mathfrak{Q}] = \Phi_t^{\langle-1\rangle} \circ \mathbf{M}_{\eta_\mathfrak{Q}(z)} \circ \Phi_t
$$
and its domain is written by
$$
{\rm Dom}\big(\mathcal{L}_r^{(t)}[\mathfrak{Q}] \bigr)=\left\{ A\in \mathcal{A}_{r}^{(t)} : \eta_\mathfrak{Q}(\cdot) \Phi_t(A)\in \mathcal{A}_r\right\},
$$
where $\mathbf{M}_{c(z)}$ is the multiplication operator by $c(z)$, that is, $\bigl(\mathbf{M}_{c(z)} A\bigr)(z)\coloneqq  c(z)A(z)$.
\end{theorem}
\begin{proof}
Since $\Phi_t$ is continuous, we get
\begin{align*}
\bigl(\Phi_t \circ \mathcal{L}_r^{(t)}[\mathfrak{Q}] \bigr)A = \lim_{s\to 0^+} \frac{\Phi_t(T_sA)-\Phi_t(A)}{s} = \lim_{s\to 0^+} \frac{\Phi_t(Q_s)-1}{s} \cdot \Phi_t(A).
\end{align*}
By assumption, we have
$$
\lim_{s\to 0^+} \frac{\Phi_t(Q_s)-1}{s} = \mathbf{M}_{\eta_{\mathfrak{Q}}(z)},
$$
which completes the proof.
\end{proof}

The above theorem shows that the $(t,r)$-deformed infinitesimal generator $\mathcal{L}_r^{(t)}[\mathfrak{Q}]$ is conjugate to the multiplication operator by $\eta_{\mathfrak{Q}}(z)$ via the transform $\Phi_t$. Thus, it is easy to compute the $(t,r)$-deformed infinitesimal generators of $\boxplus^t$-semigroups as follows.

\begin{example}
For $t\in \mathbb{R}\setminus \mathbb{Z}_{\ge0}$, we define
$$
H_s^{(t)}\coloneqq D_{\sqrt{s}}\bigl(H^{(t)}\bigr), \qquad s>0
$$ 
and set $H_0^{(t)} \coloneqq  1$, where $H^{(t)}$ denotes the Hermite series introduced in Example \ref{ex:Hermite}. It is easy to verify that $H_s^{(t)} \in \mathcal{A}_r^{(t)}$ for all $s\ge0$ and $r>0$. By Corollary \ref{cor:recursive}, one has $H_{s_1}^{(t)}\boxplus^t H_{s_2}^{(t)}=H_{s_1+s_2}^{(t)}$. By Lemma \ref{lem:SCS}, $\mathfrak{H}\coloneqq \{H_s^{(t)}\}_{s\ge0}$ is a $\boxplus^t$-semigroup. Moreover we obtain
\begin{align*}
\Phi_t\bigl(H_s^{(t)}\bigr)(z) = \sum_{k=0}^\infty \frac{1}{k!} \left(-\frac{s}{2t}\right)^k z^{2k} = \exp\left(-\frac{s}{2t}z^2\right)
\end{align*}
and therefore
$$
\eta_{\mathfrak{H}}(z)=\lim_{s\to 0}\frac{e^{-\frac{s}{2t}z^2}-1}{s} = -\frac{z^2}{2t}, \qquad z\in \mathbb{C}.
$$ 
By Theorem \ref{thm:generators}, the $(t,r)$-deformed infinitesimal generator of the semigroup $\mathfrak{H}$ is computed by
\begin{align*}
\bigl(\mathcal{L}_r^{(t)}[\mathfrak{H}] A\bigr)(z) 
&=\bigl(\Phi_t^{\langle -1\rangle} \circ M_{-\frac{z^2}{2t}} \circ \Phi_t (A)\bigr)(z)\\
&=\Phi_t^{\langle -1\rangle}\left(-\frac{1}{2t} \sum_{k=0}^\infty \frac{a_k}{(t)_{\underline{k}}} z^{k+2} \right)\\
&= -\frac{1}{2t} \sum_{k=0}^\infty \frac{(t)_{\underline{k+2}}}{(t)_{\underline{k}}} a_kz^{k+2}\\
&=-\frac{1}{2t} \sum_{k=0}^\infty (t-k)(t-k-1)a_k z^{k+2} .
\end{align*}
\end{example}

\begin{example}
Let $L_s^{(t)}$ be the $t$-deformed Laguerre series for $t \in \mathbb{R} \setminus \mathbb{Z}_{\ge 0}$ and $s > 0$ such that $st\in \mathbb{R}_{>0}\setminus \mathbb{Z}_{\ge1}$; see Example~\ref{ex:Laguerre}. Since 
$$
\sum_{k=0}^\infty\left|\frac{1}{(t)_{\underline{k}}}(-1)^k \frac{(st)_{\underline{k}}(t)_{\underline{k}}}{t^k k!}\right|r^k=\sum_{k=0}^\infty \left|\frac{(st)_{\underline{k}}}{t^k k!}\right|r^k 
$$
this converges if and only if $0<r<|t|$. Therefore $L_s^{(t)}\in \mathcal{A}_r^{(t)}$ if $0<r<|t|$.  Combining this with \eqref{eq:Lag_convolution} and Lemma \ref{lem:SCS}, the family $\mathfrak{L}=\{L_s^{(t)}\}_{s\ge0}$ forms a $\boxplus^t$-semigroup. Note that,
$$
\Phi_t\bigl(L_s^{(t)}\bigr)(z) = \sum_{k=0}^\infty(-1)^k \frac{(st)_{\underline{k}}}{t^k k!}z^k=\left(1-\frac{z}{t}\right)^{st}
$$
and therefore
$$
\eta_{\mathfrak{L}}(z)=\lim_{s\to 0} \frac{(1-z/t)^{st}-1}{s} = t \log \left(1-\frac{z}{t}\right), \quad z\in \mathbb{D}_{|t|}.
$$
Finally, by Theorem \ref{thm:generators}, we have
\begin{align*}
\bigl(\mathcal{L}_r^{(t)}[\mathfrak{L}]A\bigr)(z) &= \bigl(\Phi_t^{\langle-1\rangle} \circ \mathbf{M}_{t \log (1+z/t)} \circ \Phi_t (A)\bigr)(z)\\
&=\Phi_t^{\langle-1\rangle} \left( t \log\left(1-\frac{z}{t}\right) \sum_{k=0}^\infty \frac{a_k}{(t)_{\underline{k}}} z^k\right)\\
&=\Phi_t^{\langle-1\rangle} \left( \sum_{k=1}^\infty \left( \sum_{\ell =1}^k \frac{1}{\ell t^{\ell-1}} \frac{a_{k-\ell}}{(t)_{\underline{k-\ell}}}\right)z^k \right)\\
&= \sum_{k=1}^\infty \left( \sum_{\ell =1}^k \frac{ (t-k+\ell) \cdots (t-k+1) }{\ell t^{\ell-1}}   a_{k-\ell} \right)z^k.
\end{align*}
\end{example}

\subsection{$d$-deformed infinitesimal generators and finite free probability}\label{sec:4.3}

In finite free probability theory, the following convolution operator
\[
T_g: \mathbb{R}[x]_d \to \mathbb{R}[x]_d, \qquad f \mapsto f \boxplus_d g,
\]
was introduced and studied in connection with the differential operator which preserves real-rootedness; see \cite[Proposition 1.2]{leake2018further}. A family $\mathfrak{q} = \{q_s\}_{s \ge 0}$ in $\mathbb{R}[x]_d$ is called a $\boxplus_d$-semigroup if $q_0(x) = x^d$ and $q_{s_1} \boxplus_d q_{s_2} = q_{s_1 + s_2}$ for all $s_1, s_2 \ge 0$. For such a family $\mathfrak{q}$, we define the operator
\[
\mathcal{L}[\mathfrak{q}] f \coloneqq \lim_{s \to 0^+} \frac{T_{q_s} f - f}{s}, \qquad f \in \mathbb{R}[x]_d,
\]
whenever the limit exists. We call $\mathcal{L}[\mathfrak{q}]$ the \textit{finite free infinitesimal generator}.
The purpose of this section is to provide an explicit formula for this operator. To do this, we interpret the finite free setting in the framework of $(t,r)$-deformed infinitesimal generators in the case $t = d \in \mathbb{Z}_{\ge 1}$. 
Rather than developing a separate theory, we treat this case as a limiting instance of the general $(t,r)$-deformed framework as $t \to d$.

Note that $(\mathcal{A}_r^{(t)}, \boxplus^t) = (\operatorname{Im}\iota_d, \boxplus^d)$ for all $r > 0$ when $t = d$. 
In this case, the resulting structures are independent of $r$, and we therefore omit the parameter $r$ from the notation.
Under this identification, the $d$-deformed convolution operators and $\boxplus^d$-semigroups are defined analogously. 
Accordingly, for a $\boxplus^d$-semigroup $\mathfrak{Q} = \{Q_s\}_{s \ge 0} \subset \operatorname{Im}\iota_d$, we define the $d$-deformed infinitesimal generator by
\[
\mathcal{L}^{(d)}[\mathfrak{Q}] A 
\coloneqq \lim_{s \to 0^+} \frac{T_{Q_s} A - A}{s},
\]
whenever the limit exists. 

Finally, we obtain a formula for the finite free infinitesimal generator as follows.
\begin{theorem}
Let $\mathfrak{q} = \{q_s\}_{s \ge 0} \subset \mathbb{R}[x]_d$ be a $\boxplus_d$-semigroup. 
Then the finite free infinitesimal generator of $\mathfrak{q}$ satisfies
\begin{align*}
\iota_d\bigl(\mathcal{L}[\mathfrak{q}] f\bigr)
= \mathcal{L}^{(d)}\bigl[\iota_d(\mathfrak{q})\bigr]\bigl(\iota_d(f)\bigr) 
\qquad f \in \mathbb{R}[x]_d,
\end{align*}
where $\iota_d(\mathfrak{q}) \coloneqq \{\iota_d(q_s)\}_{s\ge0} \subset \operatorname{Im}\iota_{d}$ is a $\boxplus^d$-semigroup.
\end{theorem}

\begin{proof}
A direct computation shows that
\begin{align*}
\iota_d\left( \frac{T_{q_s}f-f}{s}\right)
&= \frac{\iota_d(T_{q_s}f)-\iota_d(f)}{s}\\
&=\frac{\iota_d(f\boxplus_d q_s)-\iota_d(f)}{s}\\
&=\frac{\iota_d(f)\boxplus^d \iota_d(q_s)-\iota_d(f)}{s} \qquad \text{(by \eqref{free-t-convolution})}\\
&=\frac{T_{\iota_d(q_s)}\bigl(\iota_d(f)\bigr) - \iota_d(f)}{s}.
\end{align*}
Taking the limit as $s \to 0^+$, we obtain the desired formula.
\end{proof}

\begin{example}
If $q_s(x) = D_{\sqrt{s}} H_d(x)$, then
\begin{align*}
\mathcal{L}^{(d)}\bigl[\iota_d(\mathfrak{q})\bigr]\bigl(\iota_d(f)\bigr)(z)
= -\frac{1}{2d} \sum_{k=0}^{d-2} (d-k)(d-k-1) f_k z^{k+2}.
\end{align*}

One can verify that the family $\{q_s\}_{s \ge 0}$ satisfies the following partial differential equation, known as the heat equation:
\[
\partial_s q_s(x) = -\frac{1}{2d} \partial_{xx} q_s(x).
\]
Equivalently, in a formal sense, we may write
\[
q_s(x) = \exp\left(-\frac{s}{2d} \partial_{xx}\right) x^d.
\]

Moreover, one can directly verify that
\[
\iota_d\!\left( -\frac{1}{2d} \partial_{xx} f \right)(z)
= -\frac{1}{2d} \sum_{k=0}^{d-2} (d-k)(d-k-1) f_k z^{k+2}.
\]

Thus, the finite free infinitesimal generator corresponds, via the embedding $\iota_d$, to the second-order differential operator $-\frac{1}{2d}\partial_{xx}$. 
In particular, this computation shows that the $d$-deformed infinitesimal generator provides an algebraic realization of the finite free heat equation.
\end{example}

\subsection{$(-1,r)$-deformed infinitesimal generators and L\'{e}vy processes}

Let $\{X(s)\}_{s \ge 0}$ be a one-dimensional L\'{e}vy process. 
Assume that the moment generating function $\Psi_{X(1)}(z)=\mathbb{E}\bigl[e^{zX(1)}\bigr]$ is finite in a neighborhood of $z=0$. 
Then it is known that, for sufficiently small $z$, 
\begin{align}\label{eq:Levy-Khintchine}
\Psi_{X(s)}(z) = \exp\bigl(s \eta_{\gamma,a,\nu}(-z) \bigr),
\end{align}
where $\eta_{\gamma,a,\nu}$ is given by
\[
\eta_{\gamma,a,\nu}(z) = -\gamma z + \frac{a}{2}z^2 + \int_{\mathbb{R}}\left(e^{-zx}-1+zx\mathbf{1}_{|x|\le 1}\right)\, \nu(dx),
\]
for $\gamma \in \mathbb{R}$, $a \ge 0$, and a L\'{e}vy measure $\nu$, that is, a positive measure on $\mathbb{R}$ satisfying $\nu(\{0\})=0$ and $\int_{\mathbb{R}} (1 \land x^2)\, \nu(dx) < \infty$. 
The formula \eqref{eq:Levy-Khintchine} is called the \textit{L\'{e}vy--Khintchine representation}\footnote{It is standard that the characteristic function $\varphi_X: \mathbb{R}\ni z\mapsto \mathbb{E}[e^{{\rm i}zX(s)}]$ admits the L\'{e}vy--Khintchine representation $\exp(s \eta_{\gamma,a,\nu}(-{\rm i}z))$. Assume that $\Psi_{X(1)}(z)$ is finite in a neighborhood of $z=0$. Then $\varphi_X$ extends to an analytic function in a neighborhood of the origin in $\mathbb{C}$, and we have $\Psi_{X(s)}(z)=\varphi_{X(s)} (-{\rm i}z) $ for sufficiently small $z$. Hence $\eqref{eq:Levy-Khintchine}$ holds true.}. 
Furthermore, the triplet $(\gamma,a,\nu)$ is uniquely determined by $\{X(s)\}_{s \ge 0}$ and is called the \textit{characteristic triplet}; see \cite{ken1999levy} for details.

\begin{lemma}\label{lem:r-condition}
Let $\{X(s)\}_{s\ge0}$ be a one-dimensional L\'{e}vy process with the characteristic triplet $(\gamma,a,\nu)$. For $r>0$, the following conditions are equivalent.
\begin{enumerate}[\rm (1)]
\item For all $s\ge0$, $\displaystyle \sum_{k=0}^\infty \left| \frac{\mathbb{E}[X(s)^k]}{k!}\right|r^k < \infty$.
\item $ \displaystyle \int_{|x|>1} e^{r|x|}\nu(dx) <\infty$.
\end{enumerate}
\end{lemma}
\begin{proof}
(1) $\Rightarrow$ (2): By the assumption, for such an $r>0$, the expectation
$$
\exp\bigl(s \eta_{\gamma,a,\nu}(-r) \bigr)=\mathbb{E}\bigl[e^{rX(s)}\bigr]= \sum_{k=0}^\infty \frac{\mathbb{E}[X(s)^k]}{k!}r^k
$$
converges. Since $\nu$ is a L\'{e}vy measure, we conclude that, for all $r>0$,
$$
I(r)\coloneqq  \int_{|x|\le 1} (e^{rx} -1 -rx ) \nu(dx) = \int_{|x|\le 1} \left( \frac{r^2}{2} x^2 + O(x^3)\right)\nu (dx) < \infty.
$$
A direct computation shows that
\begin{align*}
\exp\bigl(s \eta_{\gamma,a,\nu}(-r)\bigr)= e^{s\gamma r + \frac{sa}{2}r^2} e^{sI(r)}\exp\left(s\int_{|x|> 1}\left(e^{rx}-1\right) \nu(dx) \right).
\end{align*}
For the above exponential to converge, it is necessary that
$$
\int_{|x|>1} e^{r|x|} \nu(dx) < \infty.
$$

(2) $\Rightarrow$ (1): By assumption of the L\'{e}vy measure $\nu$, we get $\mathbb{E}[e^{rX(s)}]<\infty$ and $\mathbb{E}[e^{-rX(s)}]<\infty$. This implies that
$$
\mathbb{E}\bigl[e^{r|X(s)|}\bigr] \le \mathbb{E}\bigl[e^{rX(s)}\bigr] + \mathbb{E}\bigl[e^{-rX(s)}\bigr]<\infty.
$$
By the Schwartz inequality and Fubini's theorem, we get
\begin{align*}
\sum_{k=0}^\infty \left| \frac{\mathbb{E}[X(s)^k]}{k!}\right|r^k 
\le \sum_{k=0}^\infty \frac{\mathbb{E}[|X(s)|^k]}{k!}r^k 
= \mathbb{E}\left[ \sum_{k=0}^\infty\frac{|X(s)|^k}{k!}r^k\right] 
= \mathbb{E}\bigl[e^{r|X(s)|}\bigr]<\infty.
\end{align*}
\end{proof}

By \eqref{eq:semigroup-Levy} and Lemma \ref{lem:r-condition}, we have $M_{X(s)}\in \mathcal{A}_r^{(-1)}$ for $s\ge0$ and $r>0$ satisfying condition $(2)$ in Lemma \ref{lem:r-condition}. By Theorem \ref{Classicalresult}, we have
\begin{align}\label{eq:semigroup-Levy}
M_{X(s_1+s_2)} = M_{X(s_1)+X(s_2)} = M_{X(s_1)}\boxplus^{-1} M_{X(s_2)},
\end{align}
where the first equality follows from the independence of increments. Therefore, $\mathfrak{M}\coloneqq \{M_{X(s)}\}_{s\ge0}$ is a $\boxplus^{-1}$-semigroup by Lemma \ref{lem:SCS}. We can obtain the $(-1,r)$-deformed infinitesimal generator of $\mathfrak{M}$.

\begin{theorem}\label{thm:generator_Levy_process}
Let $\{X(s)\}_{s\ge0}$ be a one-dimensional L\'{e}vy process with the characteristic triplet $(\gamma,a,\nu)$ and $r>0$ such that $\int_{|x|>1} e^{r|x|}\nu(dx) <\infty$. Then
$$
 \mathcal{L}_r^{(-1)}[\mathfrak{M}]= \Phi_{-1}^{\langle-1\rangle} \circ \mathbf{M}_{\eta_{\gamma, a,\nu}(z)} \circ \Phi_{-1},
$$
with a domain 
$$
{\rm Dom}\bigl(\mathcal{L}_r^{(-1)}[\mathfrak{M}]\bigr)=\left\{A\in\mathcal{A}_r^{(-1)}: \eta_{\gamma,a,\nu}(\cdot) \Phi_{-1}(A) \in \mathcal{A}_r\right\}.
$$
\end{theorem}
\begin{proof}
A direct computation shows that
\begin{align*}
\Phi_{-1}\bigl(M_{X(s)}\bigr)(z)= \sum_{k=0}^\infty \frac{\mathbb{E}[X(s)^k]}{(-1)_{\underline{k}}}z^k= \sum_{k=0}^\infty (-1)^k \frac{\mathbb{E}[X_s^k]}{k!}z^k = \Psi_{-X(s)}(z).
\end{align*}
The L\'{e}vy-Khintchine representation \eqref{eq:Levy-Khintchine} implies that
$$
\Phi_{-1}\bigl(M_{X(s)}\bigr) (z)= \exp\bigl(s \eta_{\gamma,a,\nu}(z)\bigr).
$$
Since
$$
\eta_{\mathfrak{M}}(z)=\lim_{s\to 0}\frac{\exp\bigl(s \eta_{\gamma,a,\nu}(z)\bigr)-1}{s} = \eta_{\gamma,a,\nu}(z), \quad z\in \mathbb{D}_r,
$$
we obtain the desired formula by Theorem \ref{thm:generators}.
\end{proof}

The $(-1,r)$-infinitesimal generator gives rise to an evolution equation that can be interpreted as the Kolmogorov forward (or Fokker--Planck) equation for L\'{e}vy processes.

\begin{proposition}\label{cor:f-p}
Let $\{X(s)\}_{s\ge0}$ be a one-dimensional L\'{e}vy process with the characteristic triplet $(\gamma,a,\nu)$ and $r>0$ such that $\int_{|x|>1} e^{r|x|}\nu(dx) <\infty$. For $A\in \mathcal{A}_r^{(-1)}$, we define $F(s,z):=A\boxplus^{-1} M_{X(s)}$. Then $F$ satisfies the $(-1,r)$-deformed Kolmogorov forward equation:
$$
\frac{\partial}{\partial s} \Phi_{-1} \bigl(F(s,\cdot)\bigr)(z) = \eta_{\gamma, a, \nu}(z) \Phi_{-1}\bigl(F(s,\cdot)\bigr)(z), \quad \text{on} \quad \mathbb{R}_{>0} \times \mathbb{D}_r.
$$
\end{proposition}
\begin{proof}
According to a proof of Theorem \ref{thm:generator_Levy_process}, we get
\begin{align*}
\frac{\partial}{\partial s} \Phi_{-1} \bigl(F(s,\cdot)\bigr)(z)
& =\frac{\partial}{\partial s} \Phi_{-1}(A)(z) \Phi_{-1}\bigl(M_{X(s)}\bigr)(z)\\
&=\Phi_{-1}(A)(z)  \left(\frac{\partial}{\partial s} \exp\bigl(s\eta_{\gamma,a,\nu}(z)\bigr) \right)\\
&=\Phi_{-1}(A)(z) \eta_{\gamma,a,\nu}(z) \exp\bigl(s\eta_{\gamma,a,\nu}(z) \bigr) \\
&= \eta_{\gamma,a,\nu}(z) \Phi_{-1}(A)(z) \Phi_{-1}\bigl(M_{X(s)}\bigr)(z)\\
&=\eta_{\gamma,a,\nu}(z) \Phi_{-1}\bigl(F(s,\cdot)\bigr)(z).
\end{align*}
The desired result is obtained.
\end{proof}

\begin{example}
Let $\{B(s)\}_{s\ge0}$ be the one-dimensional standard Brownian motion. Its L\'{e}vy measure is the zero measure, and hence $\mathfrak{M}= \{M_{B(s)}\}_{s\ge0} \subset  \mathcal{A}_r^{(-1)}$ for all $r>0$.  and $\mathfrak{M}= \{M_{B(s)}\}_{s\ge0} \subset  \mathcal{A}_r^{(-1)}$. By Theorem \ref{thm:generator_Levy_process}, we have
$$
\mathcal{L}_r^{(-1)}[\mathfrak{M}] = \Phi_{-1}^{\langle -1\rangle} \circ \mathbf{M}_{z^2/2} \circ \Phi_{-1}.
$$ 
Moreover, by Corollary \ref{cor:f-p}, the function $F(s,z) = \bigl(A \boxplus^{-1} M_{B(s)}\bigr)(z)$ satisfies the {\it $(-1,r)$-deformed heat equation}:
\begin{align*}
\frac{\partial}{\partial s} \Phi_{-1}\bigl(F(s,\cdot)\bigr)(z)
=
\frac{z^2}{2}\,\Phi_{-1}\bigl(F(s,\cdot)\bigr)(z), \quad \text{on} \quad \mathbb{R}_{>0}\times \mathbb{C},
\end{align*}
which can be viewed as a deformation of the classical heat equation.
\end{example}

\subsection*{Acknowledgement}
ST was supported by JSPS KAKENHI Grant Number JP23H00081. 
YU was supported by JSPS KAKENHI Grant Number JP22K13925.

\bibliographystyle{amsplain}

\bibliography{bibfile}

\vspace{8mm}

\hspace{-6mm}{\bf Shuhei Tsujie}\\
Faculty of Education, Department of Mathematics, Hokkaido University of Education, 9 Hokumon-cho, Asahikawa, Hokkaido 070-8621, Japan\\
E-mail: tsujie.shuhei@a.hokkyodai.ac.jp\\

\vspace{3mm}

\hspace{-6mm}{\bf Yuki Ueda}\\
Faculty of Education, Department of Mathematics, Hokkaido University of Education, 9 Hokumon-cho, Asahikawa, Hokkaido 070-8621, Japan\\
E-mail: ueda.yuki@a.hokkyodai.ac.jp

\end{document}